\documentclass[reqno,11pt]{article}
\usepackage{latexsym}
\usepackage{amsmath}
\usepackage{amssymb}
\usepackage{bbm}
\usepackage{amsthm}
\usepackage{color}
\usepackage{dsfont}
\usepackage{mathtools}
\usepackage{xcolor}
\usepackage{listings}
\usepackage{lipsum}
\usepackage[export]{adjustbox}
\usepackage{hyperref}
\usepackage{stmaryrd}

\usepackage{graphicx}
\usepackage{epsfig}
\usepackage{catchfilebetweentags}
\usepackage[
backend=bibtex,
style=numeric, url = false, doi= false
]{biblatex}
\definecolor{Red}{rgb}{1,0,0}

 0 \DeclareMathAlphabet{\mathbfit}{OT1}{cmr}{bx}{it}

\newcommand{\bea}{\begin{eqnarray*}}
\newcommand{\eea}{\end{eqnarray*}}
\newcommand{\be}{\begin{eqnarray}}
\newcommand{\ee}{\end{eqnarray}}
\newcommand{\beq}{\begin{equation}}
\newcommand{\eeq}{\end{equation}}

\newcommand{\zd}{\mathbb{Z}^d}

\renewcommand{\P}{\mathbb{P}}
\newcommand{\E}{\mathbb{E}}
\newcommand{\N}{\mathbb{N}}

\newcommand{\Z}{\mathbb{Z}}

\newtheoremstyle{theorem}
                {10pt}
                {0pt}
                {}
                {}
                {\bfseries}
                {}
                {\newline}
                {}
\theoremstyle{theorem}
\newtheorem{thm}{Theorem}
\newtheorem{cor}[thm]{Corollary}
\newtheorem{rmk}[thm]{Remark}
\newtheorem{lem}[thm]{Lemma}

\newtheorem{prop}[thm]{Proposition}

\begin{document}

\title{The constant speed random walk on dynamical conductances} 

\author{Eszter Couillard \thanks{Department of Mathematics, Technical University of Munich, \texttt{eszter.couillard@tum.de}}, Carlo Scali \thanks{Department of Mathematics, Technical University of Munich, \texttt{carlo.scali@tum.de}}}
\date{}
\maketitle

\begin{abstract}
	In this paper we study the constant speed, biased random walk in dynamical random conductances. In certain regimes of the parameters of the model, we prove that the random walk is ballistic. We show annealed and quenched CLT in the full parameter range. Furthermore, as a consequence of the non-reversible nature of this model, we show that the classical Einstein relation does not hold. 
	\medskip
	\noindent \textbf{MSC:} 60K37 (primary), 60G50, 82C41, 60K35. 
\end{abstract}

\tableofcontents
\section{Introduction}

Random walks in dynamical random environments form an established field of research in probability theory. In these models, the random walk evolves in a random environment that itself evolves stochastically. Establishing a law of large numbers (LLN) and a central limit theorem (CLT) has been a central focus of this research. Under uniform mixing, well-established techniques, such as the construction of renewal structures, yield these results \cite{CometsZeitouni, ZeitouniQuenched, BMP1, BMP3, BMP2}. More recently, substantial efforts have focused on settings where mixing is slow or nonuniform. A non-exhaustive list of works in this direction includes \cite{AvenaBlondelFaggionato, AvenaDaCostaPeterson, BlondelHilarioTeixeira, BlondelReversible, KiousHilarioTeixeira, Conchon-KerjanKiousRodriguez}.

A central question concerns the value of the limit appearing in the LLN: if this limit is nonzero, the walk is said to be ballistic. Particularly strong progress has been made in dimension $d=1$. Recent works \cite{Conchon-KerjanKiousRodriguez} and \cite{CKP} established the existence of a limiting speed and information about its value for a broad class of models allowing slow mixing.

Less is known in higher dimensions, particularly beyond the perturbative regime, where we highlight the work \cite{BethuelsenVollering}. The perturbative regime, in which the stochasticity of the environment introduces only small perturbations to a well analysed system, is better understood.

The model studied in the present paper is the constant speed biased random walk on dynamical random conductances (CBRW). It retains the favorable properties of a fast mixing environment, allowing several results to be established with relative ease. However, the process is not reversible, making qualitative information about the limiting speed difficult to obtain. Beyond the perturbative regime, arguments typically rely on couplings. In our setting, coupling is straightforward in $d=1$, but appears more complicated in higher dimensions.

Let us now compare our model to its static-environment counterparts. One model of interest is the biased random walk on the supercritical percolation cluster. For sufficiently strong biases, the existence of a directionally transient but sub-ballistic regime was proved in \cite{SznitmanPerc, BergerGantertPeres}. Later, \cite{FriberghHammond} showed that the transition between ballistic and sub-ballistic behavior occurs at a single critical bias. In the dynamical setting, \cite{percolation} recently proved, among other precise results, that the biased variable speed random walk on dynamical percolation always has strictly positive speed in the bias direction.

For random walks in random i.i.d.\ conductances (RWRC) in the presence of an external field, \cite{Frib_conduc} proved that, in $d \ge 2$, the walk is ballistic if and only if the expected conductance is finite. The situation differs in dimension $1$, where an infinite first moment of either the conductances or the resistances causes sufficient slowdown for sub-ballistic behavior (see \cite{BergerShort, berger2020}). Recently, the first author studied the biased variable speed random walk among dynamical random conductances (VBRW) in \cite{couillard}, establishing that its speed is always strictly positive.

The intuition for why dynamics should eliminate the directionally transient but zero-speed regime is simple: slowdown is typically caused by local traps that quickly dissolve in the dynamical setting. We therefore expect our model to have positive speed throughout the parameter range. The main difference between CBRW and VBRW is that the former is not reversible, whereas the latter is, both for random conductances and on the percolation cluster. Several strategies for proving positive speed, or even directional transience, rely crucially on reversibility (see \cite{Frib_conduc, FriberghHammond, percolation, couillard}). Our model thus offers a promising setting for studying positivity of the speed: it retains many favorable features, yet established arguments do not seem to apply. In this work, we initiate its study by establishing several results on the speed and fluctuations around it, partially answering the main questions in certain parameter ranges.

\subsection{The model}

Let us consider the graph $\mathbb{Z}^d$ with nearest neighbor undirected edges, i.e.\ $e \in E(\mathbb{Z}^d)$ if and only if $e = \{x, y\}$, with $\|x - y\|_1 = 1$. We denote the neighboring relation between two vertices with the notation $x \sim y$. Let $\{\omega_t(e)\}_{t \ge 0, e \in E(\mathbb{Z}^d)}$ denote a family of non-negative weights associated to the edges and evolving in time. In the initial configuration $\{\omega_0(e)\}_{e \in E(\mathbb{Z}^d)}$ we sample the conductances to be i.i.d.\ from a probability measure $q$ on $(0, +\infty)$. We denote the set of configurations by $\Omega = (0,\infty)^{E(\zd)}$ and $Q= q^{E(\zd)}$ the product measure. Then, as the conductances evolve, each one is re-sampled from the distribution $q$ independently of everything else. The update of the edge $e \in E(\mathbb{Z}^d)$ happens according to a Poisson process of rate $\mu > 0$, which we denote $\{R(e, t)\}_{t \ge 0}$. The Poisson processes associated to the edges are all independent of each other. In general if we assume that the distribution $q$ is uniformly elliptic we write $\kappa > 1$ for the uniform elliptic constant, i.e. $q$ is supported on $[\kappa^{-1},\kappa]$. We denote by $\mathbf{P}$ the law of the environment, that is, the law generating the conductances and their time dynamics. We denote by $\mathbf{E}$ and $\mathbf{Var}$ the expectation and the variance w.r.t.\ $\mathbf{P}$. 

Let $e_1, \dots, e_d$ denote the canonical basis of $\mathbb{Z}^d$ and let us set the notation $e_{d+1}, \dots, e_{2d} = -e_1, \dots, -e_{d}$. Let $\ell \in \mathbb{S}^{d-1}$, we assume that the canonical base of $\zd$ is chosen such that $e_1 \cdot \ell \ge \dots \ge e_{d} \cdot \ell \ge 0$. For any $\lambda> 0$ and $\ell \in \mathbb{S}^{d-1}$, we define a tilted version of the conductance $e = \{x, y\}$ in the following fashion
\begin{equation*}
    \omega^{\lambda, \ell}_t(\{x, y\}) \coloneqq e^{\lambda (y - x)  \cdot \ell} \omega_t(\{x, y\}).
\end{equation*}
Let us also introduce, for $x \in \mathbb{Z}^d$
\begin{equation}\label{eqn:NormConstant}
    W^{\lambda, \ell}_t(\omega, x) \coloneqq \sum_{y \sim x} \omega^{\lambda, \ell}_t(\{x, y\}).
\end{equation}

We now define  the biased constant speed random walk on dynamical random conductances. Firstly, let $\{\mathcal{P}_t\}_{t \ge 0}$ be a Poisson process of intensity $1$ governing the jump times of the random walk. At the discontinuity points of $\{\mathcal{P}_t\}_{t \ge 0}$ that we denote by $(T_k)_{k \geq 0}$, the random walk $\{X_{t}\}_{t \ge 0}$ jumps to a neighbor with probability, for $i = 1, \dots, 2d$
\begin{equation*}
    p_{T_k}^{\lambda, \omega}(x, x + e_i) \coloneqq \frac{\omega^{\lambda, \ell}_{T_k}(\{x, x+e_i\})}{W^{\lambda, \ell}_{T_k}(\omega, x)} \qquad \text{for all $k \geq 1$, and } i = 1, \dots, 2d.
\end{equation*}
We denote by $v_{\ell}(\lambda,\mu)$ the speed of a CBRW with parameters $\lambda, \mu>0$ and $\ell \in \mathbb{S}^{d-1}$, in Lemma \ref{regenerativSpeed} we show that the speed exists and give an alternative expression for it.

\subsection{Summary of main results and structure of the paper}
In this paper we start by studying positivity of the speed for $\lambda, \mu >0$. Showing positivity for all $\lambda , \mu >0$ for dimension $d\geq 2$ turns out to be challenging as the process is not reversible, and it also can't be written as a time change of the variable speed process studied in \cite{couillard}. In this paper we show the positivity of the speed in different regimes:
\begin{itemize}
    \item For any distribution $q$, any $\mu >0$ and $\lambda$ large enough (see Theorem \ref{thm:FirstMain}).
    \item For any distribution $q$, $\mu$ large enough and $\lambda >0 $ (see Theorem \ref{thm:SecondMain}).
    \item For $q$ uniformly elliptic, $\mu$ small enough and $\lambda >0$ (see Theorem \ref{thm1}).
\end{itemize}
As we have temporal regeneration times, it is easy to show that the linear speed exists. The challenging part to prove positivity becomes to show the directional transience of the process.

Making use of the existence of the speed, we can study the CLT for the CBRW and we show in Theorem \ref{CLTVBRW} an annealed CLT for all $q$ and $\lambda, \mu >0$. The annealed CLT is then used to deduce the quenched CLT in Theorem \ref{CLTQuenched}.

At last, we look at the derivative w.r.t.\ $\lambda$ of the speed at $0$, and show that the Einstein relation does not hold for this process. This is the expected behavior as the process is not reversible, but differs significantly from the static environment case studied in \cite{NinaNagelGuo}.

\section{Preliminaries}\label{sect:Prelim}

For any $x \in \zd$ and $t \geq 0$, we denote the time and space shifted environment by \[(\theta_{x}\theta_s\omega)_t({y,z}) \coloneqq  \omega_{t+s}({x+y,x+z})\] for some $y \sim z$.
\paragraph{Regeneration times.} Let us introduce the regeneration times as in \cite{couillard}. A similar construction is also present in \cite{percolation}. Let us iteratively define the \textbf{infected set} $\{I_t\}_{t\ge0}$. We begin with $I_0 = \varnothing$. Let $\{e_i\}_{i \in \mathbb{N}}$ be an arbitrary deterministic ordering of the edges of $\mathbb{Z}^d$. Furthermore, for each edge $e_i$, we define an infinite number of copies that we denote $\{e_{i, j}\}_{j \in \mathbb{N}}$. Then, whenever the Poisson clock associated with the random walk $\{X_t\}_{t \ge 0}$ rings and assume that $X_{t^-} = z \in \mathbb{Z}^d$ we do the update
\begin{equation}\label{eqn:InfectedSet}
    I_{t} = I_{t^-} \cup \{e_{i, j}\}_{i \in \mathcal{E}(z), j \text{ minimal}}.
\end{equation}
Let us explain this last definition, $\mathcal{E}(z) \coloneqq \{e \in E(\mathbb{Z}^d) \colon z \in e\}$ and the choice of $j$ is minimal in the sense that the first time the edge $e_i$ is seen by the random walk we add $e_{i, 1}$, then at each update we add $e_{i, j}$ associated with the smallest $j \in \mathbb{N}$ such that $e_{i, j} \not \in I_{t^-}$. This is important in order to ensure that if the copy $e_{i,1}$ gets removed (we will explain how), it is added again the next time the random walks see that edge.

We remove the edges from the infected set in the following fashion. Let us define a Poisson process of inhomogeneous intensity $|I_{t}| \mu$. Each time this clock rings choose uniformly at random a (copy of an) edge from $I_t$ and remove it from the infected set. If the choice is associated to the copy $e_{i, 1}$ then we also refresh the state of the edge in the actual environment $\{\omega_t\}_{t \ge 0}$. One can easily check that in this way the joint process $(X_t, \omega_t)_{t \ge 0}$ has the correct transition rates and transition probabilities.

We are now in position to define the sequence $\{\tau_i\}_{i \in \mathbb{N}_0}$ of regeneration times with the convention $\tau_0 = 0$. Iteratively we set
\begin{equation}\label{eqn:RegTimes}
    \tau_{i + 1} \coloneqq \inf\left\{ t > \tau_i \colon I_{t} = \varnothing, I_{s} \neq \varnothing \text{ for some } s \in (\tau_i, t)\right\}.
\end{equation}
We remark that the increments $\{\tau_{i+1} - \tau_i\}_{i \ge 0}$ and $\{X_{\tau_{i+1}} - X_{\tau_i}\}_{i \ge 0}$ are two families of i.i.d.\ random variables. Recall that we denote by $(R(e,t))_{t\geq0}$ the Poisson process governing the refreshings of the edge $e \in E(\zd)$.
We define $\bar{\omega} \coloneqq (\omega_t)_{t\ge 0}$ and write $P^{\lambda}_{\bar{\omega}, x}$ for the (quenched) measure of the random walk started at $x \in \mathbb{Z}^d$ for a given environment evolution $(\omega_t)_{t\geq 0}$.
Let $\P_{\omega, x}^{\lambda,\mu}$ denote the joint (annealed) law of $(\omega_s,X_s)_{s \ge 0}$ started with configuration $\omega$ and at position $x$, and let $\P_x^{\lambda,\mu}$ denote the joint law of $(\omega_s,X_s)_{s \ge 0}$ started at $x \in \mathbb{Z}^d$ with a starting configuration distributed according to $Q$. We often omit the starting point and write $P^{\lambda}_{\bar{\omega}}$ and $\P^{\lambda,\mu}$ whenever $x = 0$. We may drop the superscript $\lambda$ and $\mu$ if they are clear form the context. The corresponding expectations are respectively denoted by $E^{\lambda}_{\bar{\omega}}, \E^{\lambda}, \E_{\omega,x}^{\lambda,\mu}$.
We write $N([\tau_n,\tau_n+1]) \coloneqq |\mathcal{P}\cap [\tau_n,\tau_n+1]|$ and $N(\tau) \coloneqq N([0,\tau_1])$. Note that for all $n \geq 1$,  $N([\tau_n,\tau_n+1])$ has the same distribution then $N(\tau)$. We summarize the properties of these random variables that we need in the following lemma.

\begin{lem}\label{lemma:ExpTail}
    There exists a constant $c = c(\mu) > 0$ such that
    \begin{equation*}
        \mathbb{E}\left[e^{c\tau}\right] < \infty, \quad \text{and} \quad \mathbb{E}\left[e^{cN(\tau)}\right] < \infty.
    \end{equation*}
\end{lem}
This lemma is an immediate consequence of the following more general lemma proved in Appendix \ref{sect:Lemma}.
\begin{lem}\label{deathbirthchain}
    Let $\alpha,\mu >0$ and $L\in \mathbb{N}$ with $L\geq 1$. Let $(A_t)_{t \geq 0 }$ be a continuous time Markov chain on $\mathbb{N}$ with the following $Q$ matrix for $i \geq 0$
    \begin{eqnarray*} 
        q(i,i+L) = \alpha,\\
        q(i+1,i) = (i+1) \mu,\\
        q(i,i) = -\alpha - i \mu.
    \end{eqnarray*}
    Assume $A_0 = 0$ a.s.\ and let $\tau = \inf\{t \geq 0: A_t = 0 \text{ and } \exists 0 \leq s \leq t: \; A_s \neq 0\}$ then $\tau$ has an exponential tail. Furthermore, for $T = \inf\{n \geq 1: A_n = 0\}$ where $(A_n)_{n \geq 0}$ is the discretization of $(A_t)_{t \geq 0 }$, we have that $T$ has an exponential tail.
\end{lem}
We can now use the regeneration times to show the existence of the speed and get a useful expression for it.
\begin{lem}\label{regenerativSpeed}
    Let $\lambda\geq 0$, $\mu >0$ and let $(X_t)_{t\geq 0}$ be a CBRW on dynamical conductances with bias in direction $\ell$, then $v_{\ell}(\lambda,\mu) = \lim_{t \rightarrow \infty} \frac{X_t\cdot \ell}{t}$ exists $\P^{\lambda,\mu}$-a.s. and 
    \[v_{\ell}(\lambda,\mu) = \frac{\E^{\lambda,\mu}[X_{\tau_1}\cdot \ell]}{\E^{\lambda,\mu}[\tau_1]} \quad \P^{\lambda,\mu}\text{-a.s.}\]
\end{lem}
\begin{proof} 
    Let us define $n(t) = \sup\{n \geq 0 : \tau_n \leq t\}$ for $t \geq 0$. We write $N(\tau_k,\tau_{k+1})$ for the number of jumps the random walk does between in the time interval $[\tau_k,\tau_{k+1}]$ and $N(\tau) := N(\tau_0,\tau_{1})$. \\
    We can then decompose $\frac{X_t \cdot \ell}{t}$ in the following way:
    \begin{eqnarray*}
        \frac{X_t \cdot \ell}{t} = \frac{n(t)}{t} \frac{X_{\tau_{n(t)}} \cdot \ell}{n(t)} + \frac{\left(X_t - X_{\tau_{n(t)}}\right) \cdot \ell}{t}
    \end{eqnarray*}
    First we study the term $\frac{\left(X_t - X_{\tau_{n(t)}}\right) \cdot \ell}{t}$. Note that for all $t \geq 0$ we have that \[\frac{\left(X_t - X_{\tau_{n(t)}}\right) \cdot \ell}{t} \leq \frac{N(\tau_{n(t)},\tau_{n(t)+1})}{t} \leq \frac{N(\tau_{n(t)},\tau_{n(t)+1})}{n(t)}.\]
    But for all $\varepsilon >0$ and for $k \in \N$ we have that 
    \begin{eqnarray*}
        \P^{\lambda,\mu}\left(\frac{N(\tau_{k},\tau_{k+1})}{k} \geq \varepsilon\right) \leq \frac{\E^{\lambda, \mu}[N(\tau)^2]}{\varepsilon^2 k^2}
    \end{eqnarray*}
    This quantity is summable so by Borel-Cantelli we have that 
    \[\lim_{t \to \infty} \frac{\left(X_t - X_{\tau_{n(t)}}\right) \cdot \ell}{t} = 0 \quad \P^{\lambda, \mu}-\text{a.s.} \]
    Next we note that the process $n(t)$ is a renewal process. This implies that 
    \begin{eqnarray*}
        \lim_{t\to \infty} \frac{n(t)}{t} = \E^{\lambda, \mu}[\tau_1]^{-1} \quad \P^{\lambda, \mu}-\text{a.s.}
    \end{eqnarray*}
    Finally, using that the regeneration times splits the process in i.i.d bits we get by the SLLN that 
    \begin{eqnarray*}
        \lim_{t\to \infty} \frac{X_{\tau_{n(t)}}\cdot \ell }{n(t)} = \E[X_{\tau_1}\cdot \ell] \quad \P^{\lambda, \mu}-\text{a.s.}
    \end{eqnarray*}
    Everything together yields the result.
\end{proof}

\begin{rmk}
	It is not possible to write the CBRW as a time change of the variable speed biased random walk on dynamical conductances (VBRW). Hence, we cannot deduce results about CBRW from the work of the first author \cite{couillard}. For completeness, we show this fact in Appendix \ref{sect:Counter}.
\end{rmk}
\begin{rmk}
    For $d=1$, $\lambda_1 >\lambda_2 >0$ it is possible to let evolve two biased random walks $(X^{\lambda_1}_t)_{t \geq 0}$ and $(X^{\lambda_2}_t)_{t \geq 0}$ on the same environment with the same jump times, such that for all $t \geq 0$, $X^{\lambda_2}_t \leq X^{\lambda_1}_t$ a.s. and $\P(X^{\lambda_2}_t < X^{\lambda_1}_t) >0$. This implies the strict monotonicity of the speed as function of the bias, and so it also implies the positivity of the speed in dimension $d=1$.
\end{rmk}

\section{Positivity of the speed}
\subsection{The speed is positive for $\lambda$ large enough}
\begin{thm}\label{thm:FirstMain}
    Let $q$ be a measure on $(0,+\infty)$ and $\mu >0$. There exists $\lambda_0 = \lambda_0(\mu, q, \ell)$ such that for all $\lambda \ge \lambda_0$ one has
    \begin{equation*}
        v_{\ell}(\lambda,\mu) >0.
    \end{equation*}
\end{thm}
\begin{proof}
    It is enough to check that for $\lambda$ large enough $\E^{\lambda, \mu}[X_{\tau} \cdot \ell] > 0$. Let $E_{\kappa}$ be the following event \[E_{\kappa} = \{ \text{ all edges seen up to $\tau_1$ are in } [\kappa^{-1},\kappa]\}.\]
    We write $\mathcal{R} = \{e_i: i \in \{1,...,2d\}, \, e_i \cdot \ell > 0\}$ and $r = \min \{l \cdot e: e \in \mathcal{R}\}$. For $\kappa > 1$, we denote
    \[\alpha = \inf \left\{\frac{\sum_{\tilde{e} \in \mathcal{R}} e^{\lambda(\tilde{e} \cdot \ell)} (\tilde{e} \cdot \ell) c_{\tilde{e}} }{\sum_{\tilde{e} \in \mathcal{R}} e^{\lambda(\tilde{e} \cdot \ell)} c_{\tilde{e}} + \kappa \sum_{\tilde{e} \notin \mathcal{R}} e^{\lambda(\tilde{e} \cdot \ell)} } - \kappa^2 \sum_{\tilde{e} \in \mathcal{R}} (\tilde{e} \cdot \ell) e^{-\lambda(\tilde{e} \cdot \ell)} : c_{\tilde{e}} \in [\kappa^{-1},\kappa] \text{ for all } \tilde{e} \in \mathcal{R}\right\},\]
    and we have
    \begin{eqnarray*}
        \E^{\lambda,\mu}[X_{\tau_1} \cdot \ell] &=& \E^{\lambda,\mu}[X_{\tau_1} \cdot \ell \mathds{1}_{E_{\kappa}}] + \E^{\lambda,\mu}[X_{\tau_1} \cdot \ell \mathds{1}_{E_{\kappa}^c}]\\
        &\geq& \alpha\E^{\lambda,\mu}[N(\tau)] - \E^{\lambda,\mu}[N(\tau) \mathds{1}_{E_{\kappa}^c}] \\
        &\geq& \alpha \E^{\lambda,\mu}[N(\tau)] - 2d q([\kappa^{-1},\kappa]^c)^{1/2}\E^{\lambda,\mu}[N(\tau)^2 ]^{1/2}.
    \end{eqnarray*} 
    We can fix $\kappa>1$ large enough such that 
    \begin{eqnarray*}
        2d q([\kappa^{-1},\kappa]^c)^{1/2}\E^{\lambda,\mu}[N(\tau)^2 ]^{1/2} \leq r/4
    \end{eqnarray*}
    Then using the following lower bound for $c \in [\kappa^{-1},\kappa]$
    \begin{eqnarray*}
        \frac{\sum_{\tilde{e} \in \mathcal{R}} e^{\lambda(\tilde{e} \cdot \ell)} (\tilde{e} \cdot \ell) c_{\tilde{e}} }{\sum_{\tilde{e} \in \mathcal{R}} e^{\lambda(\tilde{e} \cdot \ell)} c_{\tilde{e}} + \kappa \sum_{\tilde{e} \notin \mathcal{R}} e^{\lambda(\tilde{e} \cdot \ell)} } &\geq&  \frac{r }{1 + \kappa  \frac{\sum_{\tilde{e} \notin \mathcal{R}} e^{\lambda(\tilde{e} \cdot \ell)} }{\sum_{\tilde{e} \in \mathcal{R}} e^{\lambda(\tilde{e} \cdot \ell)} c_{\tilde{e}}}}\geq \frac{r }{1 + \kappa^2  \frac{\sum_{\tilde{e} \notin \mathcal{R}} e^{\lambda(\tilde{e} \cdot \ell)} }{\sum_{\tilde{e} \in \mathcal{R}} e^{\lambda(\tilde{e} \cdot \ell)}}} \\
    \end{eqnarray*}
    we can lower bound the infimum
    \begin{eqnarray*}
        \inf \left\{\frac{e^{\lambda} c - e^{-\lambda} \kappa}{e^{\lambda} c + (2d-2+e^{-\lambda}) \kappa} : c \in [\kappa^{-1},\kappa]\right\} \E^{\lambda}[N(\tau)] &\geq&  \frac{r}{1 + \kappa^2  \frac{\sum_{\tilde{e} \notin \mathcal{R}} e^{\lambda(\tilde{e} \cdot \ell)} }{\sum_{\tilde{e} \in \mathcal{R}} e^{\lambda(\tilde{e} \cdot \ell)}}} - e^{-\lambda} \kappa^2.
    \end{eqnarray*}
    This means that for $\lambda $ large enough we have that 
    \begin{eqnarray*}
        \alpha &\geq& 3r/4,
    \end{eqnarray*}
    and so for $\lambda$ large enough we obtain 
    \begin{eqnarray*}
        \E^{\lambda,\mu}[X_{\tau_1} \cdot e_1] \geq r/2 > 0.
    \end{eqnarray*}
\end{proof}
If we assume that the measure $q$ is uniform elliptic, then we can get the asymptotic monotonicity of the speed for large $\lambda$. We follow a well-known strategy introduced in \cite{BAFS} and implemented for RWRC in static environment in \cite{BergerGantertNagel}.
\begin{prop} \label{prop:Asympt}
    Let $q$ be a measure on $[\kappa^{-1},\kappa]$ then  there is $\Lambda >0$ large enough such that for all $\lambda_1 \geq \lambda_2 \geq \Lambda$ and for all $\mu>0$ \[ v_{\ell}(\lambda_1, \mu) \geq v_{\ell}(\lambda_2, \mu).\]
\end{prop}
\begin{proof}
    We adapt the proof as in \cite[Section 2]{BergerGantertNagel} to a general bias direction $\ell$. 
    For any $x \in \mathbb{Z}^d$, $t \ge 0$ and $\omega \in \Omega$, let $q_t^{\omega, \lambda}(x, 0) = 0$ and
    \begin{equation*}
    	q_t^{\omega, \lambda}(x, k) \coloneqq \sum_{j = 1}^k p_t^{\omega, \lambda}(x, x+e_j).
    \end{equation*}
    For any pair of environments $\{\omega^{(1)}_t\}_{t \ge 0}$ and $\{\omega^{(2)}_t\}_{t \ge 0}$ and biases $0<\lambda_1 \le \lambda_2$ one can construct the random walks $X^{(1)}$ and $X^{(2)}$ with the same exponential clocks. Let $\{U_n\}_{n \ge 1}$ be a family of i.i.d.\ uniform R.V.'s on $[0, 1]$. We define the coupled jump rules by
    \begin{equation*}
    	X^{(i)}_{n+1} -X^{(i)}_{n} = e_k \qquad \text{iff} \qquad q_{T_n}^{\omega^{(i)}, \lambda_i}(X_n, k-1) < U_n \le q_{T_n}^{\omega^{(i)}, \lambda_i}(X_n, k).
    \end{equation*}
Let us define the ``forward'' set $\mathbf{f}$ as the set of maximizers of $e_i \cdot \ell$ and the ``backward'' set $\mathbf{b} \coloneqq \{1, \dots, 2d\}\setminus \mathbf{f}$. Then we observe that the probability to stay in the forward set, since there exists $\delta>0$ such that $e_1 \cdot \ell - e_{|\mathbf{f}| + 1}\cdot \ell = \delta > 0$, is
\begin{align*}
	p_t^{\omega, \lambda}(\mathbf{f}) &\coloneqq \sum_{k \in \mathbf{f}} \frac{e^{\lambda e_1 \cdot \ell} \omega_t(\{x, x+e_k\})}{\sum_{k \in \mathbf{f}} \omega_t(\{x, x+e_k\})e^{\lambda e_1 \cdot \ell} + \sum_{k \in \mathbf{b}} \omega_t(\{x, x+e_k\}) e^{\lambda e_k \cdot \ell} } \\ &\ge \frac{e^{\lambda e_1 \cdot \ell} |\mathbf{f}|}{|\mathbf{f}| e^{\lambda e_1 \cdot \ell} + |\mathbf{b}| \kappa^2 e^{e_{|\mathbf{f}| + 1}\cdot \ell} } \\
	&\ge \frac{e^{\lambda \delta} |\mathbf{f}|}{|\mathbf{f}| e^{\lambda \delta} + |\mathbf{b}| \kappa^2 } \eqqcolon p^{\lambda}(\mathbf{f}).
\end{align*}
From this last expression it is clear that $p^{\lambda}(\mathbf{f}) \to 1$ sending $\lambda \to \infty$, independently of $\omega$ and $t>0$.

We now introduce the one dimensional process
\begin{equation*}
	Y_n \coloneqq \sum_{k = 0}^{n - 1}\left(2\mathds{1}\{U_k \le p^{\lambda_s}(\mathbf{f}) \} - 1\right). 
\end{equation*}
Note that if $\lambda_s$ is chosen such that $p^{\lambda_s}(\mathbf{f}) > 1/2$ then $\{Y_n\}_{n \ge 0}$ is a random walk with drift to the right. Moreover, whenever $Y_n$ jumps to the right (and $\lambda_s < \lambda_i$), then the random walk $X^{(i)}_n$ moves in some direction inside $\mathbf{f}$.

This allows us to consider the notion of \emph{super}-regeneration times $\tilde{\tau}_k, k \ge 0$ with $\tau_0 = 0$ and 
\begin{equation*}
	\tilde{\tau}_1 \coloneqq \inf\{ n \ge 2 \colon \max_{k < n-1} Y_k < Y_{n-1} < Y_n < \max_{k > n} Y_k\},
\end{equation*}
and inductively $\tilde{\tau}_{k + 1} = \tilde{\tau}_k + \tilde{\tau}_1 \circ \theta_{\tilde{\tau}_k}$, where $\theta_k$ denotes the canonical time shift by $k$ units of time.

We now observe that the increments of $Y$ times the absolute constant $f = e_1 \cdot \ell$ lower bound the increments of $X$ in the direction $\ell$. Indeed, the positive one coincide and the largest negative increment that $X\cdot \ell$ can take is $-e_1 \cdot \ell = -f$. Hence, we get that 
\begin{equation*}
	\max_{k < \tilde{\tau}_n-1} X_k\cdot \ell < X_{\tilde{\tau}_n-1}\cdot \ell < X_{\tilde{\tau}_n}\cdot \ell < \max_{k > n} X_{\tilde{\tau}_n}\cdot \ell.
\end{equation*}
Observe that the additional step to the right allows the decoupling of the environment between regeneration ``slabs''.

By established arguments (see e.g.\ \cite{SznitmanPerc}), the random variables $\{X_{\tilde{\tau}_k+1} - X_{\tilde{\tau}_k}, \tilde{\tau}_{k+1} -\tilde{\tau}_k\}_{k \ge 0}$ are independent and $\{X_{\tilde{\tau}_k+1} - X_{\tilde{\tau}_k}, \tilde{\tau}_{k+1} -\tilde{\tau}_k\}_{k \ge 1}$ are i.i.d.\ under $\mathbb{P}^{\lambda, \mu}$ and distributed as $X_{\tilde{\tau}_1}, \tilde{\tau}_1$ conditioned on the event $R \coloneqq \{Y_{n}>0, \text{ for all }n \ge 1\}$.

This allows us to find the formula for the speed
\begin{equation*}
	v(\lambda) = \frac{\mathbb{E}^{\lambda, \mu}[X_{\tilde{\tau}_1}\cdot \ell \mid R]}{\mathbb{E}^{\lambda, \mu}[\tilde{\tau}_1 \mid R]}.
\end{equation*}
We remark that once the parameters of the model are specified $\mathbb{E}^{\lambda, \mu}[\tilde{\tau}_1 \mid R]$ is explicitly computable (being a standard one-dimensional random walk computation). By the observations made above about building two walks with the same regeneration structure it is sufficient to show that 
\begin{equation*}
	\mathbb{E}\left[X^{(2)}_{\tilde{\tau}_1} - X^{(1)}_{\tilde{\tau}_1} \mid R\right]>0,
\end{equation*}
with $\lambda_1 = \lambda < \lambda + \varepsilon = \lambda_2$ and evolving in the same environment.

We are now in position to employ the arguments of \cite{BAFS}. We make the following observations:
\begin{enumerate}
	\item When $Y_n$ moves to the right both $X^{(2)}_{n}, X^{(1)}_{n}$ move in a direction in $\mathbf{f}$.
	\item When $Y_n$ moves to the left for the first time, then
	\begin{equation*}
		X^{(2)}_{n}\cdot \ell - X^{(1)}_{n}\cdot \ell \ge 0,
	\end{equation*}
as the jump happens on the same site. More strongly, with strictly positive probability $p_0>0$ we have
\begin{equation*}
	X^{(2)}_{\tilde{\tau}_1}\cdot \ell - X^{(1)}_{\tilde{\tau}_1}\cdot \ell \ge \delta.
\end{equation*}
\item If at time $\tilde{\tau}_1$ the random walk $Y_n$ has taken $k$ steps to the left then $X^{(2)}_{n}$ and $X^{(1)}_{n}$ take at most $k$ different steps.
\item If at time $\tilde{\tau}_1$ the random walk $Y_n$ has taken $k$ steps to the left then
\begin{equation*}
	X^{(2)}_{\tilde{\tau}_1}\cdot \ell - X^{(1)}_{\tilde{\tau}_1}\cdot \ell \ge - 2(k-1) f.
\end{equation*}
\item Putting these observations together, letting $\{D_k\}_{k \ge 0}$ be the event that $Y$ has taken exactly $k$ steps to the left before $\tilde{\tau}_1$ then
\begin{equation*}
	\mathbb{E}\left[X^{(2)}_{\tilde{\tau}_1}\cdot \ell - X^{(1)}_{\tilde{\tau}_1}\cdot\right] \ge p_0 \delta \mathbb{P}(D_1 \mid R) - \sum_{k \ge 2} 2(k-1) f \mathbb{P}(D_k \mid R).
\end{equation*}
Following the steps of \cite{BAFS} one can show that for all $\lambda$ large enough and $\varepsilon$ small enough the quantity on the r.h.s.\ of the last display is strictly positive.
\end{enumerate}

\end{proof}

\begin{rmk}
    Proposition \ref{prop:Asympt} shows a fundamental difference in the behavior of the constant speed compared to the variable speed process. In the uniform elliptic setting, the asymptotics of the speed of the constant speed for $\lambda$ large is increasing, but the speed of the variable speed process can be asymptotically decreasing, see \cite{couillard}.
\end{rmk}

\subsection{The speed is positive for $\mu$ large enough}
\begin{thm}\label{thm:SecondMain}
    For any distribution $q$ on $(0, \infty)$, there exists $\mu_0 = \mu_0(\lambda, q)$ such that for all $\mu \ge \mu_0$ one has that the CBRW $(X_t)_{t\geq 0}$ with bias $\lambda$ in direction $\ell$
    \begin{equation*}
        \lim_{t \to \infty} \frac{X_t \cdot \ell}{t} >0.
    \end{equation*}
\end{thm}
\begin{proof}
    As the jumps happen at constant rate 1 it is enough to look at $(\tilde{X}_n)_{n \geq 0}$ the discretization of $(X_t)_{t\geq0}$.
    Firstly, we observe that by Lemma~\ref{lemma:TechnicalInequality} below
    \begin{equation*}
        \mathbb{E}_0^{\lambda, \mu}[\tilde{X}_1 \cdot \ell] = \rho > 0.
    \end{equation*}
    Let us notice that the probability that at least one of the edges adjacent to the current position does not update before the jump is, by union bound,
    \begin{align*}
        \mathbf{P}\left( \mathrm{exp}(1) \le \max_{j = 1, \dots, 2d} \mathrm{exp}(\mu) \right) \le \sum_{i=1}^{2d}\mathbf{P}\left( \mathrm{exp}(1) \le \mathrm{exp}(\mu) \right).
    \end{align*}
    We observe that 
    \begin{equation*}
        \mathbf{P}\left(\mathrm{exp}(1) \le \mathrm{exp}(\mu)\right) = \frac{1}{\mu+1}
    \end{equation*}
    Since $2d$ is finite, for any $\varepsilon>0$ we can choose $\mu_1 = \mu_1(\varepsilon)$ such that 
    \begin{equation*}
        \mathbf{P}\left( \mathrm{exp}(1) >  \max_{j = 1, \dots, 2d} \mathrm{exp}(\mu) \right) \ge (1 - \varepsilon).
    \end{equation*}
    Let $\mathcal{A}_k \coloneqq \{\text{the environment adjacent to }\tilde{X}_{k-1} \text{ refreshes completely before }T_{k}\}$. We can split
    \begin{equation*}
        \tilde{X}_n \cdot \ell = \sum_{k = 1}^{n} (\tilde{X}_k - \tilde{X}_{k+1}) \cdot \ell \mathds{1}\{\mathcal{A}_k\} + \sum_{k = 1}^{n} (\tilde{X}_k - \tilde{X}_{k+1}) \cdot \ell \mathds{1}\{\mathcal{A}_k^c\}.
    \end{equation*}
    Then, we observe that on the event $\mathcal{A}_k$ one has
    \begin{equation*}
        \mathbb{E}^{\lambda, \mu}[(\tilde{X}_k - \tilde{X}_{k+1}) \cdot \ell \mathds{1}\{\mathcal{A}_k\} \cdot e_1] = \rho,
    \end{equation*}
    and they are all i.i.d.\ by exploiting again the memoryless property of exponential random variables. By a straightforward application of the law of large numbers 
    \begin{equation*}
        \sum_{k = 1}^{n} (\tilde{X}_k - \tilde{X}_{k+1}) \cdot \ell \mathds{1}\{\mathcal{A}_k^c\} \ge (\rho-\varepsilon)n.
    \end{equation*}
    On the other hand, since $\{|(\tilde{X}_k - \tilde{X}_{k+1})\cdot \ell|\}_{k \ge 0}$ is bounded by $1$, for $n$ large we have that
    \begin{equation*}
        \sum_{k = 1}^{n} (\tilde{X}_k - \tilde{X}_{k+1}) \cdot \ell \mathds{1}\{\mathcal{A}_k^c\} \ge -2\varepsilon n.
    \end{equation*}
    To obtain the result one can choose $\mu_0$ such that $\varepsilon \le \rho/4$.
\end{proof}

\begin{lem}\label{lemma:TechnicalInequality}
Fix $n \in \mathbb{N}$, let $X_1, \dots, X_{2n}$ be i.i.d.\ strictly positive random variables. Set $\lambda_1 \ge \dots \ge \lambda_n \ge 0$ and assume there exists $m \in \{1, \dots,  n\}$ such that $\lambda_1 \ge \dots \ge \lambda_m > 0$. For any $k \in \{1, \dots, m\}$ we have
\begin{equation*}
    E\left[ \frac{e^{\lambda_k}X_k}{\sum_{i = 1}^n e^{\lambda_i}X_i + \sum_{i = 1}^n e^{-\lambda_i}X_{n+i}}\right] > E\left[ \frac{e^{-\lambda_k}X_{n+k}}{\sum_{i = 1}^n e^{\lambda_i}X_i + \sum_{i = 1}^n e^{-\lambda_i}X_{n+i}}\right].
\end{equation*}
\end{lem}

\begin{proof}Set $k \in \{1, \dots, m\}$ as in the statement. Let us first set some notation, let $j \in \{1, \dots,n\}$
    \begin{equation}\label{eqn:ThreeQuantitiesAppendix}
        \begin{split}
            M^{(j)} &\coloneqq \sum_{i = 1}^j e^{\lambda_i}X_i + \sum_{i = 1}^j e^{-\lambda_i}X_{n+i}\\
            N^{(j)} &\coloneqq \sum_{i = j}^{n} e^{\lambda_j}\left( X_i + X_{n + i}  \right)\\
            \widetilde{N}^{(j)} &\coloneqq \sum_{i = j}^{n} e^{-\lambda_j}\left( X_i + X_{n + i}  \right).
        \end{split}
    \end{equation}
    We observe that, almost surely for $j \in \{1, \dots, m\}$
    \begin{equation*}
        N^{(j)} > \widetilde{N}^{(j)}, \quad \quad M^{(n)} - M^{(j-1)} < N^{(j)} \quad \text{ and } \quad \quad M^{(n)} - M^{(j-1)} > \widetilde{N}^{(j)}.
    \end{equation*}
    Applying these inequalities we find
    \begin{equation*}
        E\left[ \frac{e^{\lambda_k}X_k}{M^{(n)}}\right] = E\left[ \frac{e^{\lambda_k}X_k}{M^{(k-1)} + M^{(n)} - M^{(k-1)}}\right] > E\left[ \frac{e^{\lambda_k}X_k}{M^{(k-1)} + N^{(k)}}\right].
    \end{equation*}
    One obtains similarly that 
    \begin{equation*}
        E\left[ \frac{e^{-\lambda_k}X_k}{M^{(n)}}\right] < E\left[ \frac{e^{-\lambda_k}X_k}{M^{(k-1)} + \widetilde{N}^{(k)}}\right].
    \end{equation*}
    Thus, we have reduced the proof to showing that
    \begin{equation}\label{eqn:InequalityOneStepSimplified}
        E\left[ \frac{e^{\lambda_k}X_k}{M^{(k-1)} + N^{(k)}}\right] \ge E\left[ \frac{e^{-\lambda_k}X_k}{M^{(k-1)} + \widetilde{N}^{(k)}}\right].
    \end{equation}
    To this end let us make the observation that 
    \begin{equation*}
        1 = \sum_{i = k}^{n}E\left[ \frac{e^{\lambda_k}(X_i + X_{i+n})}{M^{(k-1)} + N^{(k)}} \mid M^{(k-1)}\right] + E\left[ \frac{M^{(k-1)}}{M^{(k-1)} + N^{(k)}} \mid M^{(k-1)}\right].
    \end{equation*}
    Observe that $M^{(k-1)}$ is by assumption independent of $N^{(k)}$ and its constituents, hence $\{X_i, X_{i+n}\}_{i = k, \dots, n}$ are a family of i.i.d.\ random variables also under the conditioning, in particular they are identically distributed. We deduce that
    \begin{equation*}
        E\left[ \frac{e^{\lambda_k}X_k}{M^{(k-1)} + N^{(k)}} \mid M^{(k-1)}\right] = \frac{1}{2(n - k)}\left(1 - E\left[ \frac{M^{(k-1)}}{M^{(k-1)} + N^{(k)}} \mid M^{(k-1)}\right]\right).
    \end{equation*}
    Proceeding in the same way one also obtains
    \begin{equation*}
        E\left[ \frac{e^{-\lambda_k}X_k}{M^{(k-1)} + \widetilde{N}^{(k)}} \mid M^{(k-1)}\right] = \frac{1}{2(n - k)}\left(1 - E\left[ \frac{M^{(k-1)}}{M^{(k-1)} + \widetilde{N}^{(k)}} \mid M^{(k-1)}\right]\right).
    \end{equation*}
    We recall that almost surely $N^{(k)} > \widetilde{N}^{(k)}$ which implies $ \frac{M^{(k-1)}}{M^{(k-1)} + N^{(k)}} < \frac{M^{(k-1)}}{M^{(k-1)} + \widetilde{N}^{(k)}}$ from which we deduce
    \begin{equation*}
        \frac{1}{2(n - k)}\left(1 - E\left[ \frac{M^{(k-1)}}{M^{(k-1)} + N^{(k)}} \mid M^{(k-1)}\right]\right) > \frac{1}{2(n - k)}\left(1 - E\left[ \frac{M^{(k-1)}}{M^{(k-1)} + \widetilde{N}^{(k)}} \mid M^{(k-1)}\right]\right).
    \end{equation*}
    One can conclude the proof using this last inequality together with the tower rule on both sides of \eqref{eqn:InequalityOneStepSimplified}.
\end{proof}
\subsection{The speed is positive for $\mu$ small enough}
In this section we assume that the distribution of the environment $q$ is uniformly elliptic with constant $\kappa \ge 1$. We write $\mu = 0$ for a static environment. The main result is the following.
\begin{thm}\label{thm1}
	Let $\lambda>0$. For all $\mu>0$ small enough and $e_i$ with $i \in \{1,...,2d\}$ such that $e_i \cdot \ell >0$ we have that 
	\begin{eqnarray*}
		\lim_{t\to\infty}X_t \cdot e_i=+\infty \quad \text{a.s.}
	\end{eqnarray*}
\end{thm}
\begin{cor}
    Let $\lambda>0$ then for $\mu>0$ small enough
    \begin{eqnarray*}
        v_{\ell}(\lambda,\mu) > 0.
    \end{eqnarray*}
\end{cor}
\begin{proof}
    We use Theorem \ref{thm1} and the existence of the linear speed in Lemma \ref{regenerativSpeed}.\\
    Let $e_i$ with $i \in \{1,...,2d\}$ such that $e_i \cdot \ell >0$, then 
    $(X_{\tau_n}\cdot e_i)_{n\in\mathbb{N}}$ is a one dimensional random walk with $\lim_{n\to\infty}X_{\tau_n}\cdot e_i
    =
    \lim_{t\to\infty}X_t\cdot e_i
    =
    +\infty$. It follows that
    \begin{eqnarray*}
    \mathbb{E}^{\lambda,\mu}\!\left[X_{\tau_1}\cdot e_i\right]>0.
    \end{eqnarray*}
    Using that 
    \begin{eqnarray}
        v_{\ell}(\lambda,\mu)
    = \sum_{i \in \{1,...,2d\}, \, e_i \cdot \ell > 0}
    (\ell \cdot e_i)\frac{\mathbb{E}^{\lambda,\mu}[X_{\tau_1}\cdot e_i]}
    {\mathbb{E}^{\lambda,\mu}[\tau_1]},
    \end{eqnarray}
    we deduce that 
    \[
    v_{\ell}(\lambda,\mu) > 0.\qedhere
    \] 
\end{proof}
The reminder of the section is devoted to the proof of the theorem. We can assume without loss of generality that $i = 1$.

Let $(\tilde{X}_n)_{n \geq 0}$ be the discretized random walk.
We write $H_j$ for $j \in \Z$ be the following hyperplane: $H_{j} = \{x \in \zd: x \cdot e_1 = j\}$, and $T_{j} = \inf\{n \geq 0: \tilde{X}_n \in H_j\}$ for the hitting time of that hyperplane.
\begin{lem}\label{staticLemma}
    Let $\lambda>0$, $\exists\,L\in\mathbb{N}$ depending on $\lambda,\kappa,d$ such that
\begin{eqnarray*}
\P_{\omega,0}^{\lambda,0}(T_L<T_{-L})\ge \frac23.
\qquad \forall\,\omega\in\Omega,
\end{eqnarray*}
\end{lem}
\begin{proof}
    The proof of this Lemma is in Appendix~\ref{sect:AppendixHitting}.
\end{proof}

We highlight that crucially, the statement of Lemma~\ref{staticLemma} holds uniformly over all environments $\omega \in \Omega$ and not simply with high probability.

\begin{lem}\label{Lemma 2}
    Let $\lambda>0$ then there exists an $L\in\mathbb{N}$ and a $\mu_0>0$ such that for all $ \mu<\mu_0$ and any starting configuration $\omega\in\Omega$
    \begin{eqnarray*}
    \P^{\lambda,\mu}_{\omega,0}(T_L<T_{-L})\geq \frac{3}{5}
    \end{eqnarray*}
\end{lem}
\begin{proof}
    Let us look at the discretized chain $(\widetilde{X}_n)_{n\in\mathbb{N}}$.

Let
\(
A_n=\{(x,y)\in E^d \text{ s.t. } \exists\,k<n \text{ with } X_k\in\{x,y\}\}.
\)
Then define $S$ to be the following event:
\begin{eqnarray*}
S=\{\,\forall e\in A_{N(T_L\wedge T_{-L})}: \omega(e)\text{ does not refresh in }[0,T_L\wedge T_{-L}]\,\}.
\end{eqnarray*}
So on the event $S$ the random walker behaves like a random walker on static environments until it hits one of the hyperplanes $H_l$ or $H_{-L}$. So using Lemma \ref{staticLemma} we get
\begin{eqnarray*}
\P^{\lambda,\mu}_{\omega,0}(T_L<T_{-L})
&\geq&
\P^{\lambda,\mu}_{\omega,0}(T_L<T_{-L}\cap S)
\\
&=&
\P^{\lambda,\mu}_{\omega,0}(T_L<T_{-L}\mid S)\cdot \P^{\lambda,\mu}_{\omega,0}(S)
\\
&=&
\P^{\lambda,0}_{\omega,0}(T_L<T_{-L})\cdot \P^{\lambda,\mu}_{\omega,0}(S)
\\
&\geq&
\frac23\,\P^{\lambda,\mu}_{\omega,0}(S).
\end{eqnarray*}

But for a fixed $L$ we have
$\lim_{\mu\to0}\P^{\lambda,\mu}_{\omega,0}(S)=1$
so for $\mu$ small enough $\P^{\lambda,\mu}_{\omega,0}(T_L<T_{-L})\geq\frac35.$
\end{proof}
We define 
\begin{equation*}
	T^{(k+1)}=T^{(k)}+(T_L\wedge T_{-L})\circ\theta_{T^{(k)}},\qquad \text{with} \qquad
	T^{(0)}=0\ \text{a.s.}
\end{equation*}
We write $Z_{\lambda} = \sum_{i=1}^{2d}e^{\lambda (e_i \cdot \ell)}$.
\begin{lem}\label{LemmaA} Let $l \in \N$ we have that 
    \begin{equation*}
        T^{(l)}<\infty \text{ a.s. }
    \end{equation*}
\end{lem}

\begin{proof}
    It is enough to show that
    \begin{eqnarray*}
    \P^{\lambda,\mu}_{\omega,0}(T_L\wedge T_{-L}<\infty)=1
    \qquad
    \forall \omega \in \Omega.
    \end{eqnarray*}
    Let $(\widetilde{X}_n)_{n\geq0}$ be the discretization of $(X_t)_{t\geq0}$. Then define
    \begin{eqnarray*}
    A_L=\left\{
    \widetilde{X}_{2L}-\widetilde{X}_0=2L\cdot e_1
    \right\}
    \quad\text{and}\quad
    N(T_L\wedge T_{-L})
    =\#\{\text{jumps up to }T_L\wedge T_{-L}\}.
    \end{eqnarray*}
    We observe that
    \begin{eqnarray*}
    \P^{\lambda,\mu}_{\omega,0}(A_L)
    \geq
    \left(\frac{e^{\lambda e_1 \cdot \ell}}{\kappa^2 Z_{\lambda}}\right)^{2L}
    >0.
    \end{eqnarray*}
    Hence, $N(T_L\wedge T_{-L})$ is dominated by $2L$ times a geometric random variable and
    \begin{eqnarray*}
    \P^{\lambda,\mu}_{\omega,0}(T_L\wedge T_{-L}<\infty)=1
    \qquad
    \forall \omega\in\Omega.
    \end{eqnarray*}
\end{proof}
Now we are ready to complete the proof of Theorem \ref{thm1}.
\begin{proof}[Proof of Theorem \ref{thm1}]
    Define $Y_k=X_{T^{(k)}} \cdot e_1$ and $\omega_k=\omega_{T^{(k)}}$.
    Using Lemma \ref{Lemma 2} and \ref{LemmaA}, we have that $(Y_k)_{k\geq0}$ dominates a random walk $(Z_k)_{k\geq0}$ with jump probabilities $\P(Z_{n+1}=Z_n+L)
    =
    1-\P(Z_{n+1}=Z_n-L)
    =
    \frac35.$
    
    $(Z_n)_{n\geq0}$ is transient to the right so $(Y_n)_{n\geq0}$ is also transient to the right. Using that the jump rate of the continuous random walk $(X_t)_{t\geq0}$ is constant $1$, we obtain
    \begin{eqnarray*}
    \lim_{t\to\infty}X_t \cdot e_1=+\infty
    \qquad
    \text{a.s.}
    \end{eqnarray*}
\end{proof}

\section{Central limit theorems}
\subsection{Annealed CLT}

With the same strategy employed in Lemma \ref{regenerativSpeed} (by e.g.\ running the argument with $X_t \cdot e_i$ for all $i=1, \dots, d$) one can show that 
\begin{equation*}
	\lim_{t \to \infty} \frac{X_t}{t} = \tilde{v}(\lambda, \ell, \mu) = \frac{\mathbb{E}^{\lambda, \mu}[X_{\tau_1}]}{\mathbb{E}^{\lambda, \mu}[\tau_1]}.
\end{equation*}
Observe that $\tilde{v}(\lambda, \ell, \mu)$ is a deterministic $d$-dimensional vector. One can show that the covariance matrix 
\begin{equation*}
\Sigma \coloneqq \mathbb{E}^{\lambda, \mu}\left[(X_{\tau_1} - \mathbb{E}^{\lambda, \mu}[X_{\tau_1}]) (X_{\tau_1} - \mathbb{E}^{\lambda, \mu}[X_{\tau_1}])^t \right],
\end{equation*}
is a $d\times d$ dimensional matrix of full rank. To confirm this last claim, observe that there is positive probability that before $\tau_1$ the random walk jumps only once and with strictly positive probability the jump is in any of the $2d$ possible directions. Hence, there cannot be any collinearity in the covariance matrix.

\begin{thm}[\textit{Annealed CLT}] \label{CLTVBRW} Let $\lambda \geq 0$, $\mu >0$ and $q$ be an arbitrary distribution on $(0, \infty)$. Under the annealed law $\mathbb{P}$ we have
	\begin{eqnarray} 
		\left(\frac{X_{nt}- \tilde{v}(\lambda, \ell, \mu) tn}{\sqrt{n}}\right)_{t `in [0, 1]}  \xrightarrow[t \to \infty]{(d)} (B_t)_{t \in [0, 1]}.
	\end{eqnarray}
Here $(B_t)_{t \in [0, 1]}$ denotes a centred Brownian Motion with covariance matrix $\tilde{\Sigma}\coloneqq \Sigma/\mathbb{E}^{\lambda, \mu}[\tau_1]$ and the convergence holds in Skorokhod $J_1$ topology.
\end{thm}


\begin{proof}
    The proof follows closely the proof of the annealed CLT for the variable speed random walk on dynamical percolation \cite{percolation} and dynamical conductances \cite{couillard}. 
    The $(\tau_n)_{n \geq 0}$ are regeneration times so $X_{\tau_{n}}$ is a sum of $n$ i.i.d. random variables.
    \begin{eqnarray} \label{annealedCLT}
        \frac{X_{\tau_n}- \tilde{v}(\lambda, \ell, \mu) \tau_n}{\sqrt{n}}  \xrightarrow[n \to \infty]{(d)} \mathcal{N}(0, \Sigma)
    \end{eqnarray}
    For $t \geq 0$, let $l(t)= \max\{l \in \N: \tau_l \leq t\}$. $l(t)$ is a renewal process, so by renewal argument we get that
    \begin{eqnarray} \label{eq:renewalVBRW}
        \lim_{t \rightarrow \infty} l(t)/t = \E^{\lambda, \mu}[\tau]^{-1} \quad \text{ a.s}.
    \end{eqnarray} 
    Next we decompose:
    \begin{eqnarray}
        \frac{X_t - \tilde{v}(\lambda, \ell, \mu) t}{\sqrt{t}} = \frac{X_t - X_{\tau_{l(t)}}}{\sqrt{t}} + \frac{X_{\tau_{l(t)}}  - \tilde{v}(\lambda, \ell, \mu) \E^{\lambda, \mu}[\tau_{l(t)}]}{\sqrt{t}} + \frac{\tilde{v}(\lambda, \ell, \mu) (\E^{\lambda, \mu}[\tau_{l(t)}] -t)}{\sqrt{t}}
    \end{eqnarray}
    Now note that $\|X_{t}- X_{\tau_{l(t)}}\| \leq N([\tau_{l(t)},\tau_{l(t)+1}])$, but $N([\tau_{l(t)},\tau_{l(t)+1}])$ has the same law then $N(\tau)$ and so using Lemma \ref{deathbirthchain} it has some exponential moments, For all $\varepsilon >0 $ 
    \begin{eqnarray}
        \P^{\lambda, \mu}(\|X_{t} - X_{\tau_{l(t)}}\| \geq \sqrt{t} \varepsilon) \leq \frac{\E^{\lambda, \mu}[N^2]}{t \varepsilon^2} \xrightarrow[t \to \infty]{} 0,
    \end{eqnarray}
    so
    \begin{eqnarray}
        \frac{X_t - X_{\tau_{l(t)} } }{\sqrt{t}} \xrightarrow[t \to \infty]{(d)} 0.
    \end{eqnarray}
    For the second term we have that 
    \begin{eqnarray*}
        0\leq t - \E^{\lambda, \mu}[\tau_{l(t)}] \leq \E^{\lambda, \mu}[\tau_{l(t) + 1}] - \E^{\lambda, \mu}[\tau_{l(t)}] = \E^{\lambda, \mu}[\tau_{1}],
    \end{eqnarray*}
    so 
    \begin{eqnarray}
        \frac{\tilde{v}(\lambda, \ell, \mu) (\E^{\lambda, \mu}[\tau_{l(t)}] -t)}{\sqrt{t}} \xrightarrow[t \to \infty]{(d)} 0.
    \end{eqnarray}
    Now we decompose
    \begin{eqnarray}
        \frac{X_{\tau_{l(t)}}  - \tilde{v}(\lambda, \ell, \mu) \E^{\lambda, \mu}[\tau_{l(t)}]}{\sqrt{t}} = \frac{X_{\tau_{l(t)}} - \tilde{v}(\lambda, \ell, \mu) \E^{\lambda, \mu}[\tau_{l(t)}]}{\sqrt{l(t)}} \sqrt{\frac{l(t)}{t}}.
    \end{eqnarray}
    Using that by equation \eqref{eq:renewalVBRW} that $\sqrt{\frac{l(t)}{t}} \xrightarrow[t \to \infty]{} \frac{1}{\sqrt{\E[\tau_1]}}$ a.s. 
    Using \eqref{annealedCLT}
    \begin{eqnarray}
        \frac{X_{\tau_{l(t)}} - \tilde{v}(\lambda, \ell, \mu) \E^{\lambda, \mu}[\tau_{l(t)}]}{\sqrt{l(t)}} = \frac{X_{\tau_{l(t)}} -\tilde{v}(\lambda, \ell, \mu) l(t) \E^{\lambda, \mu}[\tau_1 ]}{\sqrt{l(t)}}\xrightarrow[t \to \infty]{(d)} \mathcal{N}\left(0,\Sigma\right) .
    \end{eqnarray}
    Altogether we get 
    \begin{eqnarray}
        \frac{X_t - \tilde{v}(\lambda, \ell, \mu) t}{\sqrt{t}}  \xrightarrow[t \to \infty]{(d)} \mathcal{N}\left(0,\frac{\Sigma}{\E^{\lambda, \mu}[\tau_1]}\right).
    \end{eqnarray}
Observe that obtaining finite dimensional distribution convergence follows the same lines as above using the independence between regeneration intervals. Extending the convergence from marginals to a functional CLT requires a tightness argument which can be found in \cite[Theorem~4.1]{SznitmanCLT}. The proof relies on the existence of second moment of the regeneration times which is guaranteed in our setting by Lemma~\ref{lemma:ExpTail}, we omit the details.
\end{proof}

\subsection{Quenched CLT}

\begin{thm}[\textit{Quenched CLT}] \label{CLTQuenched} Let $\lambda \geq 0$ and $\mu >0$ and $q$ be an arbitrary distribution on $(0, \infty)$. For $\mathbf{P}$-a.e.\ environment $\omega \in \Omega$ we have that, under $P_{\bar{\omega}}^{\lambda}$
	\begin{eqnarray} 
	\left(\frac{X_{nt}- \tilde{v}(\lambda, \ell, \mu) tn}{\sqrt{n}}\right)_{t \in [0, 1]}  \xrightarrow[n \to \infty]{(d)} (B_t)_{t \in [0, 1]}.
\end{eqnarray}
Here $(B_t)_{t \in [0, 1]}$ denotes a centred Brownian Motion with covariance matrix $\tilde{\Sigma}= \Sigma/\mathbb{E}^{\lambda, \mu}[\tau_1]$ and the convergence holds in Skorohod $J_1$ topology.
\end{thm}

Some results proved here could be derived by applying the theorems of \cite{ZeitouniQuenched} (in $d>7$), by modifying the approaches of \cite{BMP1, BMP3, BMP2} and \cite{RAS}. Our approach follows \cite{ZeitouniQuenched} and extends the strategy to all dimensions under our strong mixing properties, which we believe could be somewhat relaxed.

To prove this theorem we employ a classical strategy of \cite{BolthSznit}. The section is devoted to the preparations of the proof, which is completed at the end of the section.

We define the sigma-algebra
\begin{equation}\label{eqn:TimeSigmaAlgebra}
	\widetilde{\mathcal{F}}_{n, k} \coloneqq \sigma\left\{ \omega_s(e): 0\leq s \leq k \text{ and } e \in E(\zd \cap B_1(0, n)) \right\}.
\end{equation}
Note that
\begin{equation*}
	\widetilde{\mathcal{F}}_{n, k}  = \sigma\left( \widetilde{\mathcal{F}}_{n, k-1} \cup \bigcup_{e \in B_1(0, n)} \sigma\left\{ \omega_s(e): k-1\leq s \leq k\right\} \right).
\end{equation*}
Let us consider the vertices of $\mathbb{Z}^d$ in the lexicographical ordering (for $x \in \mathbb{Z}^d$ let $p(x)$ denote the predecessor of $x$ in such ordering). We denote
\begin{equation}\label{eqn:TimeSpaceSigmaAlgebra}
	\widetilde{\mathcal{F}}_{n, k, x} \coloneqq \sigma \left( \widetilde{\mathcal{F}}_{n, k-1} \cup \bigcup_{y \in B_1(0, n), y \preceq x} \sigma\left\{ \omega_s(e): k-1\leq s \leq k \text{ and } e \sim y \right\}\right).
\end{equation}
We also set the convention $\widetilde{\mathcal{F}}_{n, k, p(0)} \equiv \widetilde{\mathcal{F}}_{n,k-1}$. 

For the $\sigma$-algebra $\widetilde{\mathcal{F}}_{m, k, x}$ let $\mathbf{E}_{k, x}^{(m)} [\cdot] = \mathbf{E}[\cdot \mid \widetilde{\mathcal{F}}_{m, k, x}]$. Let $F \colon C([0, 1], \mathbb{R}^d) \to \mathbb{R}$ be Lipschitz and bounded. We assume, without loss of generality, that both the Lipschitz constant and the supremum norm are $1$. $\{X\}_0^n$ serves as a shortcut for the path of the random walk between time $0$ and $n$ and $\{Z\}_0^n$ for the polygonal interpolation of the process
\begin{equation}\label{eqn:Scaling}
    \left\{ \frac{X_{nt} - \tilde{v}(\lambda, \ell, \mu)nt}{\sqrt{n}} \right\}_{t \in [0, 1]}.
\end{equation}

Let us introduce the event 
\begin{equation}
	E_n \coloneqq \{T_n \le 2n\}.
\end{equation}
It is easy to see that $\mathbf{P}$-almost surely we have that for some $c>0$
\begin{equation*}
P_{\bar{\omega}}^\lambda\left(T_n > 2n\right) \le e^{-cn}.
\end{equation*}
From the boundedness of $F$ we can deduce that
\begin{equation}\label{eqn:Other}
	\mathbf{Var}\left( E_{\bar{\omega}}^{\lambda}\left[ F(\{Z\}_0^n) \right] \right) \le \mathbf{Var}\left( E_{\bar{\omega}}^{\lambda}\left[ F(\{Z\}_0^n) \mathds{1}\{E_n\}\right] \right) + Ce^{-cn}.
\end{equation}
Furthermore, we observe that $E_{\bar{\omega}}^{\lambda}\left[ F(\{Z\}_0^n) \mathds{1}\{E_n\}\right]$ is $\widetilde{\mathcal{F}}_{2n}$ measurable. From now on, to shorten the notation, we fix the convention $\widetilde{\mathcal{F}}_{k, z} \equiv \widetilde{\mathcal{F}}_{2n, k, z}$ and similarly $\mathbf{E}_{k, z} \equiv \mathbf{E}_{k, z}^{(2n)}$.

\begin{lem}\label{lemma:Doob}
    We have that
    \begin{equation*}
        \mathbf{Var}\left( E_{\bar{\omega}}^{\lambda}\left[ F(\{Z\}_0^n)\mathds{1}\{E_n\}\right] \right) = \sum_{k = 1}^{2n} \sum_{z} \mathbf{E}[(\Delta_{z, k})^2],
    \end{equation*}
where we set
\begin{equation*}
    \Delta_{z, k} \coloneqq \left(\mathbf{E}_{k, z} \left[ E_{\bar{\omega}}^{\lambda} \left[ F(\{Z\}_0^n) \mathds{1}\{E_n\}\right]\right] -\mathbf{E}_{k, p(z)} \left[ E_{\bar{\omega}}^{\lambda} \left[ F(\{Z\}_0^n)\mathds{1}\{E_n\} \right]\right]\right).
\end{equation*}
\end{lem}
\begin{proof}
    The lemma follows in a standard way from the Doob's martingale construction. We observe that the necessary integrability condition follows from the boundedness of $F$.
\end{proof}

We set the notations $\mathcal{V}_z \coloneqq \{x \in \mathbb{Z}^d \colon x \sim z\}$ and recall that $\mathcal{E}_z = \{e \in E(\mathbb{Z}^d) \colon e \sim z\}$. Moreover,
\begin{equation*}
    \{X \text{ sees }z, k \} \coloneqq \{\exists s \in  [k-1, k) \colon X_s \in \mathcal{V}_z\}\cup  \left\{ \exists s \ge k \colon X_s \in \mathcal{V}_z \text{ and } R(e, s) = R(e, k) \text{ for some } e \in \mathcal{E}_z \right\} .
\end{equation*}
We recall that $R(e, s)$ is the value of the Poisson process governing the re-sampling of the conductances associated to the edge $e$ at time $s>0$. This means that, in words, $\{X \text{ sees }z, k \}$ is the event that either the walk is at the neighbourhood of $z$ during the time interval $[k-1, k)$ or after time $k$, the random walk evaluates again the edge $e$ without its value being re-sampled.

\begin{lem}\label{lem:See}
    We have that
    \begin{align*}
        \mathbf{E}_{k, z} &\left[ E_{\bar{\omega}}^{\lambda} \left[ F(\{Z\}_0^n)\mathds{1}\{E_n\} \right]\right] -\mathbf{E}_{k, p(z)} \left[ E_{\bar{\omega}}^{\lambda} \left[ F(\{Z\}_0^n)\mathds{1}\{E_n\} \right]\right] \\
        &= \mathbf{E}_{k, z} \left[ E_{\bar{\omega}}^{\lambda} \left[ F(\{Z\}_0^n)\mathds{1}\{E_n\}\mathds{1}\{X \text{ sees }z, k \} \right]\right] -\mathbf{E}_{k, p(z)} \left[ E_{\bar{\omega}}^{\lambda} \left[ F(\{Z\}_0^n)\mathds{1}\{E_n\}\mathds{1}\{X \text{ sees }z, k \} \right]\right].
    \end{align*}
\end{lem}

\begin{proof}
	Let $F(\{Z\}_0^n)$ be the indicator function that $\{X\}_0^n$ follows a $m$-long path $\pi_1, \dots, \pi_m$ and $t_1, \dots, t_m$ are jump times. Then, we observe that 
	\begin{equation*}
		E_{\bar{\omega}}^{\lambda} \left[ F(\{Z\}_0^n)\mathds{1}\{E_n\}\mathds{1}\{X \text{ sees }z, k \}^c \right]
	\end{equation*}
	does not depend on the conductances $\{\omega_{s}(e)\}_{s \in [k-1, k)}$ for $e \in \mathcal{E}_z$. Hence,
	\begin{equation*}
		\mathbf{E}_{k, z}\left[E_{\bar{\omega}}^{\lambda} \left[ F(\{Z\}_0^n)\mathds{1}\{E_n\}\mathds{1}\{X \text{ sees }z, k \}^c \right]\right] = \mathbf{E}_{k, p(z)}\left[E_{\bar{\omega}}^{\lambda} \left[ F(\{Z\}_0^n)\mathds{1}\{E_n\}\mathds{1}\{X \text{ sees }z, k \}^c \right]\right],
	\end{equation*}
	yielding the result for such $F$. For general $F$ the result follows from a standard monotone class argument.
\end{proof}
Let us introduce the random variables
\begin{equation*}
    T_{\mathrm{trs}}(n) \coloneqq \inf\{ s \ge 0 \colon E(B_1(0, n)) \text{ has been completely re-sampled}\}.
\end{equation*}
Furthermore, we set $T_{\mathrm{trs}}^0(n) = 0$ and for $j \ge 1$ $T_{\mathrm{trs}}^{j}(n) = T_{\mathrm{trs}}(n) \circ \theta_{T_{\mathrm{trs}}^{j-1}(n)} + T_{\mathrm{trs}}^{j-1}(n)$.

\begin{lem}\label{lemma:Fast}
    For all $M>0$ there exist constants $C_1, C>0$ such that
    \begin{equation*}
        \mathbf{P}\left(\sup_{k = 1, \dots, n^2} \left(T_{\mathrm{trs}}^{k}(n)  - T_{\mathrm{trs}}^{k-1}(n)\right) \ge C_1\log(n) \right) \le Cn^{-M}.
    \end{equation*}
\end{lem}

\begin{proof}
    We observe that the random variables $\{T_{\mathrm{trs}}^{k}(n)  - T_{\mathrm{trs}}^{k-1}(n)\}_{k \ge 1}$ are i.i.d. Hence, we only need to find an upper bound for the probability 
    \begin{equation*}
        \mathbf{P}\left(T_{\mathrm{trs}}^{1}(n) \ge C_1\log(n) \right),
    \end{equation*}
    and then we can adjust $C_1$. Since the clocks associated to each edge are exponentially distributed of parameter $\mu$ we obtain by union bound that 
    \begin{equation*}
        \mathbf{P}\left(T_{\mathrm{trs}}^{1}(n) \ge C_1\log(n) \right) \le 2^d n^{d} \mathbf{P}\left( \exp(\mu) \ge C_1\log(n) \right).
    \end{equation*}
    It follows easily that one can choose $C_1 = C_1(d, \mu, M)$ such that 
    \begin{equation*}
        \mathbf{P}\left( \exp(\mu) \ge C_1\log(n) \right) \le Cn^{-(M + d + 2)}.
    \end{equation*}
    The result follows by taking another union bound over $k = 1, \dots, n^2$. 
\end{proof}
We denote
\begin{equation*}
    \mathrm{Fast}(n) \coloneqq \left\{\sup_{k = 1, \dots, n^2} \left(T_{\mathrm{trs}}^{k}(n)  - T_{\mathrm{trs}}^{k-1}(n)\right) \le C_1\log(n)\right\}
\end{equation*}

Let us define
\begin{equation*}
    \tau_{k}^+ \coloneqq \inf\{\tau_k \colon \tau_k \ge T_{\mathrm{trs}}(n) \circ \theta_{k}\},
\end{equation*}
in words, this is the first regeneration time that happens after all the conductances in $B_1(0, n)$ have been re-sampled at least once after time $k$.

We also introduce the event, for $C_*>1$
\begin{equation*}
	\mathrm{Jump}_n(C_*) \coloneqq \{\tau_{n^2} \ge C_*^{-1}n^{2}\}\cap\{ \sup_{i = 1, \dots, n^2} \tau_{i} - \tau_{i-1} \le C_* \log(n)\}.
\end{equation*}
\begin{lem}\label{lemma:Jump}
	For any $M>0$ there exists $C_2>1$ large enough 
	\begin{equation*}
		\mathbb{P}_0^{\lambda, \mu}(\mathrm{Jump}_n(C_2)^c) \le Cn^{-M}.
	\end{equation*}
	Hence, almost surely for all $n$ large enough the event $\mathrm{Jump}_n(C_2)$ happens.
\end{lem}
\begin{proof}
	The proof follows from a large deviation statement, union bound and Borel-Cantelli argument. 
	
	Let us begin with the lower bound. We know that $\tau_{n^2} = \sum_{i=1}^n \tau_i - \tau_{i-1}$. Moreover, by construction, $\tau_1$ stochastically dominates an exponential of of mean $1$. We deduce that for all $k \ge 1$
	\begin{equation*}
		\mathbb{P}^{\lambda, \mu}\left(\tau_k \le k/2\right) \le e^{-I(1/2)k},
	\end{equation*}
	for $I(1/2)>0$ (this follows from standard Chernoff bound). Which readily implies that if $C_*^{-1} \le 1/2$ then 
	\begin{equation*}
		\mathbb{P}^{\lambda, \mu}\left(\tau_{n^2} \ge C_*^{-1}n^{2}\right) \le Ce^{-cn^2} \le Cn^{-M}.
	\end{equation*}
	Furthermore, we know that $\tau_1$ admits exponential moments of order $c>0$. Then, clearly
	\begin{equation*}
		\mathbb{P}^{\lambda, \mu}\left(\sup_{i = 1, \dots, n^2} \tau_{i} - \tau_{i-1} \ge C_* \log(n)\right) \le n^2 \mathbb{P}^{\lambda, \mu}\left(\tau_1 \ge C_* \log(n)\right) \le n^2e^{-cC_*\log(n)} \le Cn^{-M}.
	\end{equation*}
	By choosing $C_*=C_2>0$ large enough. The result follows by the first Borel-Cantelli lemma.
\end{proof}

For two paths $\pi_1$ and $\pi_2$ of length $m$ and $k$ respectively, let $\pi_1 \circ \pi_2$ denote their concatenation, i.e.\ for $ 0 \le s \le m+k$ we set
\begin{equation}\label{eqn:ConcatenationPaths}
    \pi_1 \circ \pi_2(s) \coloneqq 
    \begin{cases}
    \pi_1(s) \quad &\text{if } s \le m,\\
    \pi_1(m) + \pi_{2}(s - m) \quad &\text{if } s \in [ m, k+m ].
    \end{cases}
\end{equation}
Let $\{X\}_{a}^{b}$ denote the trajectory of the random walk between times $a$ and $b$, for non-negative real numbers $a\leq b$. We introduce the cut-and-glued path
\begin{equation*}
	\{\widetilde{X}\}_0^n \coloneqq \{X\}_0^{k-1} \circ \{X - X_{\tau_{k}^+}\}_{\tau_{k}^+}^{n + k - \tau_{k}^+}. 
\end{equation*}
Then, let $\{\widetilde{Z}\}_0^n$ be the polygonal interpolant obtained by plugging the path $\{\widetilde{X}\}_0^n$ into \eqref{eqn:Scaling}. Hence, we can now write the approximating quantity
\begin{equation*}
    F(\{\widetilde{Z}\}_0^n\mathds{1}\{E_n\}).
\end{equation*}

\begin{lem}
    $\mathbf{P}$-a.s. we have that
    \begin{align*}
        \mathbf{E}_{k, z} \left[ E_{\bar{\omega}}^{\lambda} \left[ F(\{\widetilde{Z}\}_0^n) \mathds{1}\{E_n\}\mathds{1}\{X \text{ sees }z,k\} \right]\right] =\mathbf{E}_{k, p(z)} \left[ E_{\bar{\omega}}^{\lambda} \left[ F(\{\widetilde{Z}\}_0^n)\mathds{1}\{E_n\} \mathds{1}\{X \text{ sees }z,k\}\right]\right].
    \end{align*}
\end{lem}

\begin{proof}
	Let $\widetilde{G}_j, j = 1, 2$ here denote an arbitrary measurable function of (distinct portions of) the trajectory. Observe that
	\begin{enumerate}
		\item  The trajectory $ \{X - X_{\tau_{k}^+}\}_{\tau_{k}^+}^{n + k - \tau_{k}^+}$ is independent of the value taken by the environment before time $k$. Hence, 
		\[E_{\bar{\omega}}^{\lambda}\left[ \widetilde{G}_1\left( \{X - X_{\tau_{k}^+}\}_{\tau_{k}^+}^{n + k - \tau_{k}^+}\right)\right]\] 
		is independent of $\widetilde{\mathcal{F}}_{k, z}$ for all $z \in B_1(0, n)$. Thus, its value is the same under $\mathbf{E}_{k, z}$ and $\mathbf{E}_{k, p(z)}$.
		\item The trajectory $\{X\}_0^{k-1}$ is generated by only knowing the environment information contained in $\widetilde{\mathcal{F}}_{k-1}$. Hence, 
		\[E_{\bar{\omega}}^{\lambda} \left[ \widetilde{G}_2(\{X\}_0^{k-1})\right]\]
		the same under both $\widetilde{\mathcal{F}}_{k, z}$ and $\widetilde{\mathcal{F}}_{k, p(z)}$ as it only depends on $\widetilde{\mathcal{F}}_{k-1}$.
		\item The indicator function $\mathds{1}\{X \text{ sees }z,k\}$ and its quenched expectation do not depend on the environment at $\mathcal{E}_z$ in the time interval $[k-1, k]$.
	\end{enumerate} 
	We also notice that we can write $F\left(\{\widetilde{Z}\}_0^{n}\right)$ as an increasing limit of linear combinations of functions of the type
	\begin{equation*}
		\widetilde{G}_2(\{X\}_0^{k-1}) \widetilde{G}_1\left(\{X - X_{\tau_{k}^+}\}_{\tau_{k}^+}^{n + k -\tau_{k}^+}\right).
	\end{equation*}
	We observe that $\mathds{1}\{E_n\}$ does not depend on the environment at $z$. We deduce the result by taking first expectations w.r.t.\ $E_0^\omega$ and then applying the enumerated observations to deduce the equality under $\mathbf{E}_{k, z}$ and $\mathbf{E}_{k, p(z)}$.
\end{proof}

\begin{lem}\label{lemma:Lipschitz}
    For all $M>0$ there exists $C>0$ such that $\mathbf{P}$-a.s.
    \begin{equation*}
    	\begin{split}
        \Bigg|\mathbf{E}_{k, z} \bigg[ E_{\bar{\omega}}^{\lambda} &\left[ F(\{Z\}_0^n)\mathds{1}\{E_n\} \mathds{1}\{X \text{ sees }z,k\} \right]\bigg]- \mathbf{E}_{k, z} \left[ E_{\bar{\omega}}^{\lambda} \left[ F(\{\widetilde{Z}\}_0^n)\mathds{1}\{E_n\} \mathds{1}\{X \text{ sees }z,k\} \right]\right]\Bigg| \\& \le \frac{C\log(n)^2}{\sqrt{n}}\mathbf{E}_{k, z}\left[ E_{\bar{\omega}}^{\lambda} \left[ \mathds{1}\{X \text{ sees }z,k\} \right]\right] + Cn^{-M}.
    \end{split}
    \end{equation*}
    The same fact holds under $\mathbf{E}_{k, p(z)}$ on both sides.
\end{lem}
\begin{proof}
	Firstly we observe that by boundedness of $F$
	\begin{equation*}
	\mathbf{E}_{k, z} \bigg[ E_{\bar{\omega}}^{\lambda} \left[ F(\{Z\}_0^n)\mathds{1}\{E_n\} \mathds{1}\{X \text{ sees }z,k\} \mathds{1}\{\mathrm{Fast}(n) \cap \mathrm{Jump}_n(C_2)\}^c \right]\bigg] \le \mathbf{E}_{k, z} \bigg[ P_{\bar{\omega}}^{\lambda}(\{\mathrm{Fast}(n) \cap \mathrm{Jump}_n(C_2)\}^c)\bigg].
\end{equation*}
	Then by Lemma\ref{lemma:Fast} and Lemma\ref{lemma:Jump} we obtain that 
	\begin{equation*}
		\mathbf{E}\left[\mathbf{E}_{k, z} \left[ P_{\bar{\omega}}^{\lambda}(\{\mathrm{Fast}(n) \cap \mathrm{Jump}_n(C_2)\}^c)\right]\right] \le Cn^{-2M},
	\end{equation*}
	Markov's inequality and the first Borel-Cantelli lemma and the triangle inequality yield the $Cn^{-M}$ term in the statement.
	
    We observe that on the event $\mathrm{Fast}(n) \cap \mathrm{Jump}_n(C_2)$
    \begin{equation*}
        \left\| \{Z\}_0^n - \{\widetilde{Z}\}_0^n \right\|_\infty \le \frac{C_1C_2\log(n)^2}{\sqrt{n}}.
    \end{equation*}
    Indeed, the difference between the two paths (seen as continuous functions by the interpolation) is at most of size $C_1C_2\log(n)^2$ and the scaling by $\sqrt{n}$ yields the claim.

    The result follows by applying the Lipschitz property of $F$.
\end{proof}

\begin{lem}\label{lem:VarRed}
    For all $M>1$, there exists $C>0$, such that
    \begin{equation*}
        \mathbf{Var}\left( E_{\bar{\omega}}^{\lambda}\left[ F(\{Z\}_0^n)\mathds{1}\{E_n\}\right] \right) \le \sum_{k = 1}^{2n} \sum_{z \in B_1(0, n)} \frac{C\log(n)^4}{n}\mathbf{E}\left[ E_{\bar{\omega}}^{\lambda}[\mathds{1}\{X \text{ sees }z,k\}]^2 \right] + Cn^{-M}.
    \end{equation*}
\end{lem}
\begin{proof}
    Plugging in the result of Lemma~\ref{lemma:Lipschitz} into the decomposition of Lemma~\ref{lemma:Doob} and using Lemma~\ref{lem:See} we see that we need to bound
    \begin{equation*}
        \mathbf{E}\left[ \left(\frac{2C\log(n)^2}{\sqrt{n}}\mathbf{E}_{k, z}\left[ E_{\bar{\omega}}^{\lambda} \left[ \mathds{1}\{X \text{ sees }z,k\} \right]\right] +  \frac{2C\log(n)^2}{\sqrt{n}}\mathbf{E}_{k, p(z)}\left[ E_{\bar{\omega}}^{\lambda} \left[ \mathds{1}\{X \text{ sees }z,k\} \right]\right]\right)^2\right]
    \end{equation*}
    By Jensen's inequality, for both $\mathbf{E}_{k, z}$ and $\mathbf{E}_{k, p(z)}$ (we only write for the former), we observe that 
    \begin{align*}
        \mathbf{E}\left[ \left(\frac{2C\log(n)^2}{\sqrt{n}}\mathbf{E}_{k, z}\left[ E_{\bar{\omega}}^{\lambda} \left[ \mathds{1}\{X \text{ sees }z,k\} \right]\right]\right)^2\right] &\le \frac{4C^2 \log(n)^4}{n}\mathbf{E}\left[ \mathbf{E}_{k, z}\left[ E_{\bar{\omega}}^{\lambda} \left[ \mathds{1}\{X \text{ sees }z,k\} \right]^2\right]\right]\\
        & = \frac{4C^2 \log(n)^4}{n}\mathbf{E}\left[ E_{\bar{\omega}}^{\lambda} \left[ \mathds{1}\{X \text{ sees }z,k\} \right]^2\right],
    \end{align*}
    where we applied the tower property in the final equality. Applying the standard fact $(a+b)^2 \le 2a^2 + 2 b^2$ (and possibly adjusting the constant $C$) yields the result.
\end{proof}

Let $X^{(1)}, X^{(2)}$ be two copies of the random walk $X$. We consider them, firstly, under the joint quenched measure defined in the following way. For $x, y \in \mathbb{Z}^d$ and two measurable sets $A, B$
\begin{equation*}
	P^{\lambda, \bar{\omega}}_{x, y} \left( X^{(1)}_\cdot \in A, X^{(2)}_\cdot \in B \right)  \coloneqq P^{\lambda}_{\bar{\omega}, x} (X^{(1)}_\cdot \in A) \times P^{\lambda}_{\bar{\omega}, y} (X^{(2)}_\cdot \in B).
\end{equation*}
We define $E_{x, y}^{\lambda, \bar{\omega}}$ to be the corresponding expectation. Then, taking expectation w.r.t.\ the environment law (including the dynamics) $\mathbf{P}$, we define the annealed law of two walks evolving independently in the same environment as $\mathbb{P}_{x, y}^{\lambda, \mu}(\cdot) \coloneqq \mathbf{E}[P^{\lambda, \bar{\omega}}_{x, y}(\cdot)]$

We introduce the \emph{joint regeneration times}, let $\sigma_0 = 0$ and for $i \ge 0$
\begin{equation}\label{eqn:RegTimesCommon}
    \sigma_{i + 1} \coloneqq \inf\left\{ t > \sigma_i \colon I^{(1)}_{t} \cup I^{(2)}_{t} = \varnothing, I^{(1)}_{s} \cup I^{(2)}_{s} \neq \varnothing \text{ for some } s \in (\sigma_i, t)\right\}.
\end{equation}
The fact that the $\sigma_i - \sigma_{i-1}, i \ge 1$ are well defined and with exponential tails follows from the same strategy we used in Lemma~\ref{lemma:ExpTail}. Moreover, we can define the equivalent for $\sigma_i$ of the event $\mathrm{Jump}_n(C_2)$ and prove the equivalent of Lemma~\ref{lemma:Fast}. We denote this event by $\mathrm{Jump}^*_n(C_2)$

Let $A, B \subset \mathbb{Z}^d$, we set the notation
\begin{equation*}
    \mathrm{dist}\left(A , B \right) \coloneqq \inf_{x \in A, y \in B} \|x-y\|.
\end{equation*}

\begin{lem}\label{lem:ToCrossings}
 We obtain
    \begin{align*}
        \sum_{k = 1}^{2n} \sum_{z \in B_1(0, n)} \mathbf{E}\left[ E_{\bar{\omega}}^{\lambda}[\mathds{1}\{X \text{ sees }z,k\}]^2 \right] &\le C\log(n)^{c} \sum_{j = 1}^{2n} \mathbf{E}\left[ E_{0, 0}^{\lambda, \bar{\omega}}[\mathds{1}\{\mathrm{dist}\left(\{X^{(1)}\}_{\sigma_j}^{\sigma_{j+1}} , \{X^{(2)}\}_{\sigma_j}^{\sigma_{j+1}} \right) \le 1\}]\right] \\ 
        &+ Cn^{-M}.
    \end{align*} 
\end{lem}

\begin{proof}
     We observe that
    \begin{equation*}
        \mathbf{E}\left[ E_{\bar{\omega}}^{\lambda} \left[ \mathds{1}\{X \text{ sees }z,k\} \right]^2\right] = \mathbf{E}\left[ E_{0, 0}^{\lambda, \bar{\omega}} \left[ \mathds{1}\{X^{(1)} \text{ sees }z,k\} \mathds{1}\{X^{(2)} \text{ sees }z,k\} \right]\right].
    \end{equation*}
    Observe that, by definition, a version of an edge can be ``seen'' only in one (joint) regeneration period. Furthermore, a version of an edge can be seen by both walks during a regeneration period only if the event
    \begin{equation*}
        \left\{\mathrm{dist}\left(\{X^{(1)}\}_{\sigma_j}^{\sigma_{j+1}} , \{X^{(2)}\}_{\sigma_j}^{\sigma_{j+1}} \right) \le 1\right\}
    \end{equation*}
    happens. Further, we observe that on the event $\mathrm{Fast}(n)\cap\mathrm{Jump}^*_n(C_2)$ there are at most $C\log(n)^2$ meetings in each regeneration period. Hence, 
    \begin{align*}
        \sum_{k = 1}^{2n} \sum_{z \in B_1(0, n)}&\mathbf{E}\left[ E_{\bar{\omega}}^{\lambda}[\mathds{1}\{X \text{ sees }z,k\}]^2 \right]\\
        &\le \sum_{j = 1}^{2n} \sum_{z \in B_1(0, n)}\mathbf{E}\left[ E_{0, 0}^{\lambda, \bar{\omega}}[\mathds{1}\{\{X^{(1)}\}_{\sigma_j}^{\sigma_{j+1}} \text{ sees }z\} \mathds{1}\{\{X^{(2)}\}_{\sigma_j}^{\sigma_{j+1}} \text{ sees }z\}] \right]\\
        &\le C\log(n)^2  \sum_{j = 1}^{2n} \mathbf{E}\left[ E_{0, 0}^{\lambda, \bar{\omega}}[\mathds{1}\{\mathrm{dist}\left(\{X^{(1)}\}_{\sigma_j}^{\sigma_{j+1}} , \{X^{(2)}\}_{\sigma_j}^{\sigma_{j+1}} \right) \le 1\}\}]\right].
    \end{align*}
    Finally, since 
    \begin{equation*}
    	\mathbb{P}_0^{\lambda, \mu}\left(\{\mathrm{Fast}(n)\cap\mathrm{Jump}^*_n(C_2)\}^c\right) \le Cn^{-M},
    \end{equation*}
    we obtain
    \begin{equation*}
    	\sum_{k = 1}^{2n} \sum_{z \in B_1(0, n)}\mathbf{E}\left[ E_{\bar{\omega}}^{\lambda}[\mathds{1}\{\mathrm{Fast}(n)\cap\mathrm{Jump}_n^*(C_2)\}^c]^2 \right] \le Cn^{-M+d+2}.
    \end{equation*}
    Since we can choose $M$ arbitrarily large we deduce the result.
\end{proof}

Let $\mathbb{P}_x^{\lambda, \mu}\times\mathbb{P}_y^{\lambda, \mu}$ denote the (annealed) law induced by two random walks evolving in independent environments. We denote expectations and variances with respect to this measure in the obvious way.

Let us introduce the random quantity
\begin{equation*}
	\mathrm{Col} \coloneqq \inf \left\{ i \ge 1 \colon N(\sigma_{i} - \sigma_{i-1})+1 \ge \| X^{(1)}_{\sigma_{i-1}} -X^{(2)}_{\sigma_{i-1}} \|  \right\}.
\end{equation*}

\begin{lem}\label{lem:StayingFar}
    For any $\varepsilon>0$, there exists $\Lambda>0$, $c>0$ such that uniformly over all $x, y \in \mathbb{Z}^d$ with $\|x-y\| \ge \Lambda$ 
    \begin{equation*}
        \mathbb{P}_x^{\lambda, \mu}\times\mathbb{P}_y^{\lambda, \mu}\left( \mathrm{Col} \ge n \right) \ge \frac{c}{n^{1/2+ \varepsilon}}.
    \end{equation*}
\end{lem}
\begin{proof}
    For any $M>0$, we observe that there exists $U > 0$ and $C>0$ such that for all $n \ge 1$
    \begin{equation*}
        \mathbb{P}_x^{\lambda, \mu}\times\mathbb{P}_y^{\lambda, \mu}\left( N(\sigma_{n} - \sigma_{n-1}) + 1 \ge U \vee C\log(n) \right) \le \frac{1}{n^{M}},
    \end{equation*}
    where $U$ is only required for $n$ small.
    On the other hand, for $\eta \in (0, 1/2)$, let us introduce the event
    \begin{equation*}
        \left\{ \| X^{(1)}_{\sigma_n} - X^{(2)}_{\sigma_n} - (x - y)  \| \ge n^{\eta} \right\}.
    \end{equation*}
    We observe that on the intersection  
    \begin{equation*}
        \left\{ \| X^{(1)}_{\sigma_n} - X^{(2)}_{\sigma_n} - (x - y)  \| \ge n^{\eta} \right\} \cap \left\{ N(\sigma_{n} - \sigma_{n-1}) \le U \vee C\log(n) \right\},
    \end{equation*}
    there must exist some $n_0 > 0$ such that the event
    \begin{equation*}
        \left\{ \forall n \ge n_0 \colon N(\sigma_{n} - \sigma_{n-1}) + 1 \le \| X^{(1)}_{\sigma_{n-1}} -X^{(2)}_{\sigma_{n-1}} - (x - y) \| \right\},
    \end{equation*}
    happens. But then, we can choose $\Lambda>0$ large enough such that
    \begin{equation*}
        \left\{ \forall n \ge 1 \colon N(\sigma_{n} - \sigma_{n-1}) + 1 \le \| X^{(1)}_{\sigma_{n-1}} -X^{(2)}_{\sigma_{n-1}} \| \right\},
    \end{equation*}
    happens whenever $\|x- y\| \ge \Lambda $ and $X_0^{(1)} = x$ and $X_0^{(2)} = y$. Then proving the statement reduces to proving that
    \begin{equation*}
        \mathbb{P}_0^{\lambda, \mu}\times\mathbb{P}_0^{\lambda, \mu}\left( \| X^{(1)}_{\sigma_j} - X^{(2)}_{\sigma_j}\| \ge j^{\eta}, N(\sigma_{j} - \sigma_{j-1}) \le U \vee C\log(j), \text{ for all } j \in [1, n] \right) \ge \frac{c}{n^{1/2+\varepsilon}}.
    \end{equation*}
    Observe that $\{X^{(1)}_{\sigma_j} - X^{(2)}_{\sigma_j}\}$ is a mean zero random walk on $\mathbb{Z}^d$. A direct application of Lemma~4.3 of \cite{Berger_Zeitouni} yields the desired result.
\end{proof}

\begin{lem}\label{lem:ProductJoint}
    For all $x, y \in \mathbb{Z}^d$ with $\|x-y\| \ge \Lambda$ 
    \begin{equation*}
        \mathbb{P}_{x, y}^{\lambda, \mu}\left( \mathrm{Col} \ge n \right) = \mathbb{P}^{\lambda, \mu}_x\times\mathbb{P}^{\lambda, \mu}_y\left( \mathrm{Col} \ge n \right).
    \end{equation*}
\end{lem}

\begin{proof} 
    For any point $x \in \mathbb{Z}^d$ let us introduce $\mathcal{E}_x \coloneqq \{e \in E(\mathbb{Z}^d) \colon e \sim x\}$, i.e.\ the set of edge incident to $x$. For a path $\pi = \pi(T_i)_{i=0}^n$ (with corresponding jump times) we denote 
    \begin{equation*}
        \mathcal{E}_\pi \coloneqq \{c(e, t) \colon \exists i =1, \dots, n \text{ s.t. } c(e, T_i) \sim \pi(T_i)\}.
    \end{equation*}
    We observe that the event $\{\mathrm{Col} \ge n\}$ is measurable with respect to the sigma field
    \begin{equation*}
        \mathcal{G} \coloneqq \sigma \left\{\{X^{(1)}\}_{0}^{\sigma_n} \cup \{X^{(2)}\}_{0}^{\sigma_n} \cup \mathcal{E}_{\{X^{(1)}\}_{0}^{\sigma_n}} \cup \mathcal{E}_{\{X^{(2)}\}_{0}^{\sigma_n}} \right\}.
    \end{equation*}
    On the event $\{\mathrm{Col} \ge n\}$ the conductances seen by the two random walks are disjoint, indeed, this is a consequence of our definition of regeneration times. The disjointedness and independence of the conductances implies that
    \begin{equation*}
        \mathbb{P}_{x, y}^{\lambda, \mu}\left( \mathrm{Col} \ge n \right) = \mathbb{P}^{\lambda, \mu}_x\times\mathbb{P}^{\lambda, \mu}_y\left( \mathrm{Col} \ge n \right),
    \end{equation*}
    as required.
\end{proof}

\begin{lem}\label{lem:Lambda}

    For any $\Lambda>0$ there exists $\rho = \rho(\Lambda)>0$ all $x, y \in \mathbb{Z}^d$  
    \begin{equation*}
        \mathbb{P}_{x, y}^{\lambda, \mu}\left( \| X^{(1)}_{\sigma_{1}} -X^{(2)}_{\sigma_{1}} \| \ge \Lambda \right) \ge \rho >0 .
    \end{equation*}
\end{lem}

\begin{proof}
    We construct an event on which the two random walks get far away from each other. In particular, let us fix a constant $\kappa > 1$, moreover, let us fix two nearest neighbor paths $\pi^{(1)}(i)_{i = 1}^{\ell}$ and $\pi^{(2)}(i)_{i = 1}^{\ell}$ chosen (in some minimal way) such that $\|\pi^{(1)}(\ell) - \pi^{(2)}(\ell)\| \ge \Lambda$. We ask the random walk to follow these paths and encounter at the jump point only conductances of size $c(e) \in [\kappa^{-1}, \kappa]$. It follows from the fact that the jumps of Poisson processes admit density that
    \begin{equation*}
        \mathbb{P}_{x, y}^{\lambda, \mu}\left( \{ X^{(1)}\}_0^{\sigma_1} = \pi^{(1)}(\ell),  \{X^{(2)}\}_{0}^{\sigma_{1}} = \pi^{(2)}(\ell) \right) \ge \rho>0.
    \end{equation*}
    This concludes the proof. We highlight that the probability $\rho$ decays exponentially in $\ell$ and hence in $\Lambda$, however, since we only aim to show the result for finite $\Lambda$ this is sufficient for our purpose.
\end{proof}

	Let us define the random variable 
\begin{equation*}
	J^* \coloneqq \inf\{j \ge 1 \colon \mathrm{dist}\left(\{X^{(1)}\}_{\sigma_j}^{\sigma_{j+1}} , \{X^{(2)}\}_{\sigma_j}^{\sigma_{j+1}} \right) \le 1\}.
\end{equation*}

\begin{cor}\label{cor:Useful}
	For any $x, y \in \mathbb{Z}^d$ and any $\varepsilon>0$ there exists a constant $c>0$ such that
	\begin{equation*}
		\mathbb{P}_{x, y}^{\lambda, \mu} \left(J^*\ge n \right) \ge \frac{c}{n^{1/2 + \varepsilon}}.
	\end{equation*}
\end{cor}
\begin{proof}
	Observe that by the strong Markov property we obtain
	\begin{equation*}
		\mathbb{P}_{x, y}^{\lambda, \mu} \left(J^* \ge n \right) \ge \mathbb{P}_{x, y}^{\lambda, \mu}\left( \| X^{(1)}_{\sigma_{1}} -X^{(2)}_{\sigma_{1}} \| \ge \Lambda \right) \inf_{v, w \in \mathbb{Z}^d \colon \|v-w\| \ge \Lambda}\mathbb{P}_{v, w}^{\lambda, \mu}\left( \mathrm{Col} \ge n \right).
	\end{equation*}
	Here, we observe that $\{\mathrm{Col} \ge n\} \subseteq \{\inf\{j \ge 2 \colon \mathrm{dist}\left(\{X^{(1)}\}_{\sigma_j}^{\sigma_{j+1}} , \{X^{(2)}\}_{\sigma_j}^{\sigma_{j+1}} \right) \le 1\} \ge n\}$. 
	
	The quantitative estimate follows by plugging in the r.h.s.\ the results of Lemma~\ref{lem:StayingFar}, Lemma~\ref{lem:ProductJoint} and Lemma~\ref{lem:Lambda}.
\end{proof}

\begin{prop}\label{prop:Crossings}
	For all $d \ge 1$, for all $\varepsilon>0$ there exists $C = C(d, \varepsilon) > 0$ such that
	\begin{equation*}
		\mathbf{E}\left[ E_{0, 0}^{\lambda, \bar{\omega}}\left[\sum_{j = 1}^{2n}\mathds{1}\{\mathrm{dist}\left(\{X^{(1)}\}_{\sigma_j}^{\sigma_{j+1}} , \{X^{(2)}\}_{\sigma_j}^{\sigma_{j+1}} \right) \le 1\}\right]\right] \le Cn^{1/2+ \varepsilon}.
	\end{equation*}
\end{prop}

\begin{proof}
	Let us define the increasing sequence of random indices, $\psi_0 = 0$ and then $\psi_{k+1} = J^* \circ \theta_{\sigma_{\psi_k}+1, \sigma_{\psi_k}+1} + \psi_k$.
	
	Observe that for all $k\ge 0$ in the time interval $[\sigma_{\psi_k}, \sigma_{\psi_{k+1}}]$ there are at most two $j \ge 0$ such that $\mathrm{dist}\left(\{X^{(1)}\}_{\sigma_j}^{\sigma_{j+1}} , \{X^{(2)}\}_{\sigma_j}^{\sigma_{j+1}} \right) \le 1$. We deduce that
	\begin{equation*}
		\sum_{j = 1}^{2n}\mathds{1}\{\mathrm{dist}\left(\{X^{(1)}\}_{\sigma_j}^{\sigma_{j+1}} , \{X^{(2)}\}_{\sigma_j}^{\sigma_{j+1}} \right) \le 1\} \le 2 \sup\{j \ge 0 \colon \psi_j \le 2n \}.
	\end{equation*}
	Note that at every $\psi_k, k \ge 0$ we have by Corollary~\ref{cor:Useful} that, uniformly over $X^{(1)}_{\sigma_{\psi_k}}, X^{(2)}_{\sigma_{\psi_k}}$
	\begin{equation*}
		\mathbb{P}^{\lambda, \mu}_{X^{(1)}_{\sigma_{\psi_k}}, X^{(2)}_{\sigma_{\psi_k}}} \left( \psi_{k+1} - \psi_k \ge 2n \right) \ge \frac{c}{n^{1/2 + \varepsilon}}.
	\end{equation*}
	Note that on the event $\psi_{k+1} - \psi_k \ge 2n$ we have $\sup\{j \ge 0 \colon \psi_j \le 2n \} \le k$.
	Hence, by the independence induced by the strong Markov property we obtain
	\begin{align*}
		\mathbb{P}_{0, 0}^{\lambda, \mu}\left(\sum_{j = 1}^{2n}\mathds{1}\{\mathrm{dist}\left(\{X^{(1)}\}_{\sigma_j}^{\sigma_{j+1}} , \{X^{(2)}\}_{\sigma_j}^{\sigma_{j+1}} \right) \le 1\} \ge 2 n^{1/2 + 2\varepsilon}\right) &\le \left(1 - \frac{c}{n^{1/2 + \varepsilon}}\right)^{n^{1/2 + 2\varepsilon}}\\
		& \le e^{-cn^\varepsilon}.
	\end{align*}
	Finally, we obtain
	\begin{align*}
		\mathbf{E}&\left[ E_{0, 0}^{\lambda, \bar{\omega}}\left[\sum_{j = 1}^{2n}\mathds{1}\{\mathrm{dist}\left(\{X^{(1)}\}_{\sigma_j}^{\sigma_{j+1}} , \{X^{(2)}\}_{\sigma_j}^{\sigma_{j+1}} \right) \le 1\}\right]\right] =
			\mathbb{E}_{0, 0}^{\lambda, \mu}\left[\sum_{j = 1}^{2n}\mathds{1}\{\mathrm{dist}\left(\{X^{(1)}\}_{\sigma_j}^{\sigma_{j+1}} , \{X^{(2)}\}_{\sigma_j}^{\sigma_{j+1}} \right) \le 1\}\right]\\
			&\le n^{1/2 + 2\varepsilon} + n\mathbb{P}^{\lambda, \mu}_{0, 0}\left(\sum_{j = 1}^{2n}\mathds{1}\{\mathrm{dist}\left(\{X^{(1)}\}_{\sigma_j}^{\sigma_{j+1}} , \{X^{(2)}\}_{\sigma_j}^{\sigma_{j+1}} \right) \le 1\} \ge n^{1/2 + 2\varepsilon}\right)\\
			& \le n^{1/2 + 2 \varepsilon} + n e^{-cn^\varepsilon}\\
			& \le Cn^{1/2 + 2\varepsilon}.
	\end{align*}
The proof is concluded as $\varepsilon> 0$ is arbitrary.
\end{proof}

\begin{proof}[Proof of Theorem~\ref{CLTQuenched}]
	Putting together \eqref{eqn:Other}, Lemma~\ref{lem:VarRed}, Lemma~\ref{lem:ToCrossings} and Proposition~\ref{prop:Crossings} we obtain
	\begin{equation*}
	\mathbf{Var}\left( E_{\bar{\omega}}^{\lambda}\left[ F(\{Z\}_0^n)\right] \right) \le Cn^{-c}.
	\end{equation*}
	This is nowadays a classical condition that, together with the annealed functional CLT Theorem~\ref{CLTVBRW} implies the quenched functional CLT, see \cite[Lemma~4.1]{Szn_serie}.
\end{proof}
\section{Einstein relation does not hold}\label{sec:Einstein}
In the following we write $d^{\lambda, \ell}(\omega,x) = \frac{\sum_{i=1}^d (\ell \cdot e_i) \left(e^{\lambda(\ell \cdot e_i)} \omega(x , x +e_i) - e^{-\lambda (\ell \cdot e_i)} \omega(x , x -e_i)\right) }{W^{\lambda}(\omega,x)}$ for the quenched local displacement in direction $\ell$ at $x$ in the environment $\omega$.\\
We also allow for $\lambda$ to be negative. If $\lambda <0$ we get the same process as a CBRW with bias $|\lambda|$ in direction $-\ell$. This allows us to look at the derivative at $0$ instead of the right derivative at $0$.
\begin{lem}
For any $\mu > 0$ we have 
\[\frac{d}{d\lambda} v_{\ell}(\lambda,\mu) \Big|_{\lambda = 0}  = \frac{1}{\E^{0,\mu}[\tau]} \left(\E^{0,\mu}[(X_{\tau} \cdot \ell)^2] - \E^{0,\mu}\Big[(X_{\tau} \cdot \ell) \sum_{k=1}^{N(\tau)} d^{0,\ell}(\omega_k,X_{k-1})\Big]\right).\]
\end{lem}

\begin{proof}
In the following, we denote by $(\tilde{X}_k)_{k \in \mathbb{N}}$ the discrete embedded chain in $(X_t)_{t \geq 0}$ and the corresponding environments will be denoted by $(\tilde{\omega}_k)_{k \in \N}$.
First we recall that $\P$-a.s. $v(\lambda,\mu) =\frac{\E^{\lambda, \mu}[X_{\tau}\cdot \ell]}{\E^{\lambda, \mu}[\tau]} $ and that $\E^{\lambda, \mu}[\tau]$ does not depend on $\lambda$.\\
We want to compute $\frac{d}{d\lambda} \E^{\lambda,\mu}[X_{\tau}\cdot \ell] $. Note that
\begin{equation}
  \E^{\lambda,\mu}[X_{\tau}\cdot \ell] = \sum_{n = 1}^{\infty}\E^{\lambda,\mu}[X_{\tau}\cdot \ell  \mathds{1}_{N(\tau) = n}].
\end{equation}
Hence, we first show that we can exchange the derivative operation with the infinite sum and with the integral (expectation). We need to show that there exists $\lambda_0>0$ such that 
\begin{enumerate}
    \item $\sum_{n = 1}^{\infty} \frac{d}{d\lambda}\E^{\lambda,\mu}[X_{\tau}\cdot \ell  \mathds{1}_{N(\tau) = n}]$ converges uniformly over $\lambda \in [-\lambda_0, \lambda_0]$. This implies that we can swap the series and the derivative.
    \item For each $n \in \mathbb{N}$  we can apply the Leibniz integral rule on $(-\lambda_0,\lambda_0)$ to swap the derivative and the expectation. 
\end{enumerate}
We start by showing the second point by choosing a (positive) function which does depend on $n$, say $g(\lambda, n)$. Then we show that such function can be chosen such that for any $\varepsilon>0$ and uniformly over $\lambda \in [-\lambda_0, \lambda_0]$ there exists $m>0$ such that
\begin{equation*}
    \sum_{n = m}^{\infty} \frac{d}{d\lambda}\E^{\lambda,\mu}[g(\lambda, n)  \mathds{1}_{N(\tau) = n}] \le \varepsilon.
\end{equation*}
Let $\mathcal{P}_n$ be the set of all the paths of length $n$ starting in $0$. Then $|\mathcal{P}_n| < \infty$, and 
\begin{eqnarray*}
    \frac{d}{d\lambda} \E^{\lambda,\mu}[X_{\tau}\cdot \ell \mathds{1}_{N(\tau) = n}] &=& \sum_{p \in \mathcal{P}_n} \frac{d}{d\lambda} \E^{\lambda, \mu}[X_{\tau}\cdot \ell \mathds{1}_{N(\tau) = n}\mathds{1}_{(\tilde{X}_0,...,\tilde{X}_n) = p} ]\\
    &=&\sum_{p \in \mathcal{P}_n} \frac{d}{d\lambda} p_n\cdot \ell \E^{0,\mu}\left[ e^{\lambda p_n \cdot e_1} \prod_{k = 1}^{n} \frac{W^0(\tilde{\omega}_k,p_{k-1})}{W^{\lambda}(\tilde{\omega}_k,p_{k-1})}\;\mathds{1}_{N(\tau) = n}\mathds{1}_{(\tilde{X}_0,...,\tilde{X}_n) = p}  \right]\\
    &=&\sum_{p \in \mathcal{P}_n} p_n\cdot \ell \frac{d}{d\lambda} \E^{0,\mu}\left[ e^{\lambda p_n \cdot \ell + \sum_{k = 1}^{n} \ln\left(\frac{W^0(\tilde{\omega}_k,p_{k-1})}{W^{\lambda}(\tilde{\omega}_k,p_{k-1})}\right) }\;\mathds{1}_{N(\tau) = n}\mathds{1}_{(X_0,...,X_n) = p}  \right].
\end{eqnarray*}
Next we compute $\frac{d}{d\lambda}e^{\lambda p_n \cdot \ell + \sum_{k = 1}^{n} \ln\left(\frac{W^0(\tilde{\omega}_k,p_{k-1})}{W^{\lambda}(\tilde{\omega}_k,p_{k-1})}\right) }$
\begin{eqnarray*}
    &&\frac{d}{d\lambda}e^{\lambda p_n \cdot \ell + \sum_{k = 1}^{n} \ln\left(\frac{W^0(\tilde{\omega}_k,p_{k-1})}{W^{\lambda}(\tilde{\omega}_k,p_{k-1})}\right) } \\
    &&= \left(p_n \cdot \ell - \sum_{k = 1}^{n} \frac{\sum_{i=1}^d (\ell \cdot e_i) \left(e^{\lambda(\ell \cdot e_i)} \tilde{\omega}_k(p_{k-1} , p_{k-1} +e_i) - e^{-\lambda (\ell \cdot e_i)} \tilde{\omega}_k(p_{k-1} , p_{k-1} -e_i)\right) }{W^{\lambda}(\tilde{\omega}_k,p_{k-1})} \right)\\
    && \quad e^{\lambda p_n \cdot \ell + \sum_{k = 1}^{n} \ln\left(\frac{W^0(\tilde{\omega}_k,p_{k-1})}{W^{\lambda}(\tilde{\omega}_k,p_{k-1})}\right) }.
\end{eqnarray*}
We observe that for any $x \in \mathbb{Z}^d$ we have that 
\begin{equation*}
    W^0(\tilde{\omega}_k,p_{k-1})e^{-|\lambda|} \le W^\lambda(\tilde{\omega}_k,p_{k-1}) \le W^0(\tilde{\omega}_k,p_{k-1})e^{|\lambda|}.
\end{equation*}
Which in turn implies
\begin{eqnarray*}
    \Big|\frac{d}{d\lambda}e^{\lambda p_n \cdot \ell + \sum_{k = 1}^{n} \ln\left(\frac{W^0(\tilde{\omega}_k,p_{k-1})}{W^{\lambda}(\tilde{\omega}_k,p_{k-1})}\right) } \Big| \leq (p_n \cdot \ell + n ) e^{|\lambda| p_n \cdot \ell +  n |\lambda| } \le 2n e^{2|\lambda| n}.
\end{eqnarray*}
The quantity on the right-hand side is bounded above for $|\lambda| \in [0, 1]$ by $2n e^{2 n}$ which is bounded (for fixed $n$) and thus integrable.

A fortiori, we have shown that 
\begin{align*}
    \sum_{n = m}^\infty \frac{d}{d\lambda} \E^{\lambda,\mu}[X_{\tau}\cdot e_1  \mathds{1}_{N(\tau) = n}] &\le \sum_{n = m}^\infty \mathbb{E}^{0, \mu}\left[ 2n^2 e^{2|\lambda| n} \mathds{1}\{N(\tau) = n\}\right] \\
    & \le \mathbb{E}^{0,\mu}\left[ 2N(\tau)^2 e^{2|\lambda| N(\tau)} \mathds{1}\{N(\tau) \ge m\}\right].
\end{align*}
By applying twice the Cauchy-Schwarz inequality (note also that $N(\tau)$ is integer valued) we deduce that
\begin{align*}
    \mathbb{E}^{0,\mu}\left[ 2N(\tau)^2 e^{2\lambda N(\tau)} \mathds{1}\{N(\tau) \ge m\}\right] \le 16 \mathbb{E}^{0,\mu}\left[ N(\tau)^8\right] \mathbb{E}^{0,\mu}\left[e^{16|\lambda| N(\tau)} \right] \mathbb{P}^{0,\mu}\left( N(\tau) \ge m\right)^{1/2}.
\end{align*}
However, we know that there exists $\rho > 0$ depending only on $\mu$ (this is the crucial point of the argument) such that 
\begin{equation*}
    \mathbb{E}^{0,\mu}\left[e^{\rho N(\tau)} \right] < \infty.
\end{equation*}
Hence, there exists a constant $C$ uniform in $|\lambda|\in [0, \rho/16]$ such that 
\begin{equation*}
    16 \mathbb{E}^{0,\mu}\left[ N(\tau)^8\right] \mathbb{E}^{0,\mu}\left[e^{16|\lambda N(\tau)|} \right] \le C.
\end{equation*}
However, this implies that one can choose $m>0$ depending only on $\varepsilon$ and not on $\lambda \in [-\rho/16, \rho/16]$ such that 
\begin{equation*}
    \sum_{n = m}^\infty \frac{d}{d\lambda} \E^{\lambda,\mu}[X_{\tau}\cdot \ell  \mathds{1}_{N(\tau) = n}] \le \varepsilon,
\end{equation*}
as desired.

The computations above justify the fact that for every $\lambda \in [-\lambda_0, \lambda_0]$ (with $\lambda_0 = \rho/16$) we have that
\begin{align*}
    &\frac{d}{d\lambda} v_{\ell}(\lambda,\mu)\\ & = \frac{1}{\E^{0,\mu}[\tau]} \sum_{n = 1}^\infty \sum_{p \in \mathcal{P}_n} p_n \cdot \ell \mathbb{E}^{0, \mu} \Bigg[ \left(p_n \cdot \ell - \sum_{k = 1}^{n} \frac{\sum_{i=1}^d (\ell \cdot e_i) \left(e^{\lambda(\ell \cdot e_i)} \tilde{\omega}_k(p_{k-1} , p_{k-1} +e_i) - e^{-\lambda (\ell \cdot e_i)} \tilde{\omega}_k(p_{k-1} , p_{k-1} -e_i)\right) }{W^{\lambda}(\tilde{\omega}_k,p_{k-1})} \right)\\
    & \quad \cdot e^{\lambda p_n \cdot \ell + \sum_{k = 1}^{n} \ln\left(\frac{W^0(\tilde{\omega}_k,p_{k-1})}{W^{\lambda}(\tilde{\omega}_k,p_{k-1})}\right) } \mathds{1}\{N(\tau) = n\} \mathds{1}\{(\tilde{X}_0,...,\tilde{X}_n) = p\}\Bigg].
\end{align*}
And finally, plugging in $\lambda=0$ we obtain
\begin{equation*}
	 \frac{d}{d\lambda} v_{\ell}(\lambda,\mu) \vert_{\lambda=0}=\frac{1}{\E^{0,\mu}[\tau]} \left(\E^{0,\mu}[(X_{\tau} \cdot \ell)^2] - \E^{0,\mu}[X_{\tau} \cdot \ell \cdot \sum_{k=1}^{N(\tau)} d^{0,\ell}(\tilde{\omega}_k,\tilde{X}_{k-1})]\right).
\end{equation*}
\end{proof}
    \begin{rmk}
        Note that the same argument can be done at any point $\lambda_0$ different from $0$ as the derivative in $\lambda$ still can be upper-bounded by an integrable quantity
        \begin{eqnarray*}
        \Big|\frac{d}{d\lambda}e^{(\lambda-\lambda_0) p_n \cdot \ell + \sum_{k = 1}^{n} \ln\left(\frac{W^{\lambda_0}(\tilde{\omega}_k,p_{k-1})}{W^{\lambda}(\tilde{\omega}_k,p_{k-1})}\right) } \Big| \leq (|1 - \lambda_0| p_n \cdot \ell + n ) e^{(\lambda -\lambda_0) p_n \cdot \ell +  n (\lambda -\lambda_0) } \le 2n e^{2(\lambda -\lambda_0) n}.
    \end{eqnarray*}
    This implies that the function $v_{\ell}(\lambda,\mu)$ is differentiable.
\end{rmk} 

\begin{thm}
    There exists $\mu >0$ such that 
    \[\frac{d}{d\lambda} v(\lambda,\mu) \Big|_{\lambda = 0}  \neq \frac{\E^{0, \mu}[(X_{\tau} \cdot \ell)^2]}{\E^{0, \mu}[\tau]} \quad \P\text{-a.s.}\]
    Using the CLT (Theorem \ref{annealedCLT}), we deduce that the classical statement of the Einstein relation does not hold.
\end{thm}

\begin{proof}
    It is enough to show that for some $\mu >0$, $\E^0[X_{\tau} \cdot \ell \cdot \sum_{k=1}^{N(\tau)} d^{0,\ell}(\tilde{\omega}_k,\tilde{X}_{k-1})] \neq 0$.\\
    Let $R_1= \{N(\tau) = 1\}$ then event that the random walk does exactly one step in the regeneration time.
    Let again $(\tilde X_k)_{k\geq 0}$ be the discretized random walk and $(\tilde \omega_k)_{k \geq 0}$ the discretized environment.
    \begin{eqnarray*}
        \E^{0, \mu}\Bigg[X_{\tau} \cdot \ell \cdot \sum_{k=1}^{N(\tau)} d^{0,\ell}(\tilde{\omega}_k,\tilde{X}_{k-1})\Bigg] &=&  \E^{0, \mu}\left[\tilde X_1 \cdot \ell \cdot  d^{0,\ell}(\tilde \omega_1,\tilde X_{0}) \mathds{1}_{R_1}\right] + \E^{0, \mu}\Bigg[X_{\tau} \cdot \ell \cdot \sum_{k=1}^{N(\tau)} d^{0,\ell}(\omega_k,X_{k-1}) \mathds{1}_{R_1^c}\Bigg]\\
        &\geq& \E^{0, \mu}\left[\tilde X_1 \cdot \ell \cdot  d^{0,\ell}(\tilde \omega_1,\tilde X_{0}) \mathds{1}_{R_1}\right] - \E^{0, \mu}\left[N(\tau)^2 \mathds{1}_{R_1^c}\right].
    \end{eqnarray*}
    We observe that $R_1$ is independent of $\tilde X_1 \cdot \ell \cdot  d^{0,\ell}(\tilde \omega_1,\tilde X_{0})$, so we get 
    \begin{eqnarray*}
        \E^{0, \mu}\Bigg[X_{\tau} \cdot \ell \cdot \sum_{k=1}^{N(\tau)} d^{0,\ell}(\tilde{\omega}_k,\tilde{X}_{k-1})\Bigg] &\geq &  \P^{0, \mu}(R_1) \E^{0, \mu}\left[\tilde X_1 \cdot \ell \cdot  d^{0,\ell}(\tilde \omega_1,0)\right] - \E^{0, \mu}\left[N(\tau)^2 \mathds{1}_{R_1^c}\right]. 
    \end{eqnarray*}
    Then using that $\E^{0, \mu}[N(\tau)^2 ] < \infty$ and that $\lim_{\mu \rightarrow \infty} \mathds{1}_{R_1^c} = 1$ we get that for any  $\varepsilon > 0$ there is a $\mu$ large enough such that 
    \begin{eqnarray*}
        \E^{0, \mu}\Bigg[X_{\tau} \cdot \ell \cdot \sum_{k=1}^{N(\tau)} d^{0,\ell}(\tilde{\omega}_k,\tilde{X}_{k-1})\Bigg] &\geq & \E^{0, \mu}\left[\tilde X_1 \cdot \ell \cdot  d^{0,\ell}(\tilde \omega_1,0)\right] - \varepsilon  \\
        &=&\E^{0, \mu}\left[\E[X_{1} \cdot \ell|\tilde{\omega}_1]^2\right] - \varepsilon.
    \end{eqnarray*}
    So for $\mu$ large enough we have that 
    \begin{eqnarray*}
        \E^{0, \mu}\Bigg[X_{\tau} \cdot \ell \cdot \sum_{k=1}^{N(\tau)} d^{0,\ell}(\tilde{\omega}_k,\tilde{X}_{k-1})\Bigg]>0.
    \end{eqnarray*}
\end{proof}
\begin{rmk}
    Even though the Einstein relation does not hold, we believe that $\frac{d}{d\lambda} v_{\ell}(\lambda,\mu) \Big|_{\lambda = 0} > 0$, and so that the speed is positive for small $\lambda$.
\end{rmk}

\appendix 

\section{Probability of hitting hyperplanes} \label{sect:AppendixHitting}

In this section we prove Lemma \ref{staticLemma}. The proof adapts the argument of \cite[Lemma 5]{NinaNagelGuo}. We assume that $\ell \cdot e_1 >0$, and we write $\lambda_1 = \lambda (\ell \cdot e_1)$.\\
Fix $d$ the dimension, $q$ the distribution with uniform elliptic constant $\kappa>0$, so $\Omega=[\kappa^{-1},\kappa]^{E(\zd)}$.\\
Form now on, as in \cite{NinaNagelGuo}, we use the conductances $\omega^\lambda(x,y)=e^{\lambda(x+y)\cdot \ell}\omega(x,y)$ on the edges to give the jump probabilities of the random walker. Note that this yields the same model because the jump rate of the random walker is constant $1$ and the probability to jump to some neighbor does not change. Indeed, for $x \sim y$
\begin{eqnarray*}
\P_{\omega,z}^{\lambda,0}(X_{n+1}=y\,|\,X_n=x)
&=&
\frac{e^{\lambda(x+y)\cdot \ell}}
{\sum_{z\sim x}e^{\lambda(x+z)\cdot \ell}\omega(x,z)}
\\
&=&
\frac{e^{2\lambda x\cdot \ell}}
{e^{2\lambda x\cdot \ell}}
\frac{e^{\lambda(y-x)\cdot \ell}}
{\sum_{z\sim x}e^{\lambda(z-x)\cdot e_1}\omega(x,z)}
\omega(x,y).
\end{eqnarray*}

\bigskip

\begin{rmk}
In \cite{NinaNagelGuo}, the authors prove the existence of $L$ uniformly in $\lambda$ which we do not need, this way we do not need to use the Harnack inequality.
\end{rmk}

\textbf{Proof of Lemma \ref{staticLemma}:} We argue by contradiction. Fix $L\ge1$. Let $u(x)=\P_{\omega,x}^{\lambda,0}(T_L<T_{-L})$ and assume there exists $\omega\in\Omega$ s.t. $u(0)<2/3$.

Let $T_{mL}=\{z\in\mathbb Z^d:\;z\cdot e_1=mL\}$ for $m\in\mathbb Z$. For a set $G\subset\mathbb R^d$ we define its discrete boundary by $\partial G=\{x\in\mathbb Z^d\cap G:\;x\sim y\text{ for some }y\in\mathbb Z^d\setminus G\}$.
For a function $h:G\to\mathbb R$ denote its Dirichlet energy on $G$ as
\begin{eqnarray*}
    \mathcal E(h,G)=\sum_{\substack{x,y\in G\\x\sim y}}\omega^\lambda(x,y)(h(x)-h(y))^2.
\end{eqnarray*}
We let $S_G=S_G(u,\omega)$ be the set of functions $v:G\to\mathbb R$ with $v=u$ on $\partial G$. We check that $u$ is harmonic on $\{x\in\mathbb Z^d:\;|x\cdot e_1|\le L\}$.
Let $x$ be such that $|x\cdot e_1|<L$. Then
\begin{eqnarray*}
    u(x)&=&\P_{\omega,x}^{\lambda,0}(T_L<T_{-L})\\
    &=&\sum_{y\sim x}\P_{\omega,x}^{\lambda,0}(T_L<T_{-L}\mid X_1=y)\P_{\omega,x}^{\lambda,0}(X_1=y)\\
    &\stackrel{\mathrm{Markov property}}{=}&\sum_{y\sim x}\P_{\omega,y}^{\lambda,0}(T_L<T_{-L})\P_{\omega,x}^{\lambda,0}(X_1=y)\\
    &=&\sum_yu(y)\P_{\omega,x}^{\lambda,0}(X_1=y).
\end{eqnarray*}

Let $x$ be such that $|x\cdot e_1|=L$ then $\P_{\omega,x}^{\lambda,0}(T_L<T_{-L})=1_{\{x\cdot e_1=L\}}$.\\
Then applying Dirichlet principle we have
\begin{eqnarray*}
    \mathcal E(u,G)=\min_{v\in S_G}\mathcal E(v,G)
\end{eqnarray*}
for any $G\subseteq\{x\in\mathbb Z^d:\;|x\cdot e_1|\le L\}$.

\bigskip

\textit{Lower bound:} 
We start by finding a lower bound on $\mathcal E(u,\Pi_{l})$ for $l \geq 1$ and $\Pi_{l}=[-L,L]\times[-l,l]^{d-1}$.
We do that by ignoring edges that are not in direction $e_1$. For $x= (x_1,...,x_d) \in \Z^d$ we write $\bar{x}$ for $(x_2,...,x_d)$.

The conductance of such an edge connecting $(x_1,0)$ to $(x_1+1,0)$ (with $0\in\mathbb Z^{d-1}$) is bounded from below by $\kappa^{-1} e^{2\lambda_1 x_1}$, a lower bound for $\mathcal E(u,\Pi_{l})$ is then
\begin{eqnarray*}
    \mathcal E(u,\Pi_{l})&\ge& \kappa^{-1}\sum_{i=-L}^{L-1}e^{2\lambda_1 i}\left(u(i,0)-u(i+1,0)\right)^2\\
    &\ge&\kappa^{-1}\sum_{i=0}^{L-1}e^{2\lambda_1 i}\left(u(i,0)-u(i+1,0)\right)^2\\
    &\ge&\kappa^{-1}\left(\sum_{i=0}^{L-1}u(i,0)-u(i+1,0)\right)^2\left(\sum_{i=0}^{L-1}e^{-2\lambda_1 i}\right)^{-1}\\
    &\ge&\kappa^{-1}\frac{1-e^{-2 \lambda_1}}{1- e^{-2\lambda_1 L}}\left(1-u(0)\right)^2\\
    &\ge&\kappa^{-1}(1-e^{-2 \lambda_1})(1-u(0,\bar x))^2=c_1>0,
\end{eqnarray*}
We applied in the third line Cauchy-Schwarz and $c_1$ is a strictly positive constant independent of $L$.\bigskip \\

\textit{Upper bound:} Next we want to look at $\bar u(x)=\dfrac{e^{2L\lambda_1}-e^{-2x_1\lambda_1}}{e^{2L\lambda_1}-e^{-2L\lambda_1}}$ with $x_1=x\cdot e_1$.
For $x\cdot e_1=L$ we have $\bar u(x)=1$ and for $x\cdot e_1=-L$ we have $\bar u(x)=0$.
Further $\bar u(x)\geq 0$ $\forall\,x$ s.t.\ $|x\cdot e_1|\leq L$.
Let $\tilde L=\lfloor e^{L\lambda_1 /d}\rfloor\cdot L$.
Let
$\Pi=[-L,L]\times[-2\tilde L,2\tilde L]^{d-1}$
and
$\tilde\Pi=[-L,L]\times[-\tilde L,\tilde L]^{d-1}$.
We first want to compute $\mathcal E(\bar u,\tilde{\Pi})$.
\begin{eqnarray*}
    \mathcal E(\bar u,\tilde{\Pi}) = \sum_{\substack{y\nsim x\\x,y\in\Pi}} \P^{\lambda,0}_{\omega,x}(X_1=y) (\bar u(x)-\bar u(y))^2.
\end{eqnarray*}
We see that for $(x-y)\cdot e_1=0$ we have $\bar u(x)-\bar u(y)=0$, so we can focus on $(x-y)\cdot e_1\neq 0$. We have
\begin{eqnarray*}
    \mathcal E(\bar u,\tilde{\Pi}) &=& \sum_{\bar x \in \Z^{d-1}, \, \|\bar x\|\le 2\tilde L} \sum_{i=-L}^{L-1} \P^{\lambda,0}_{\omega,(i,\bar x)} \bigl(X_1=(i+1,\bar x)\bigr) \left(\bar u(i+1,\bar x)-\bar u(i,\bar x)\right)^2\\
    &&+\sum_{\bar x \in \Z^{d-1}, \, \|\bar x\|\le 2\tilde L}\sum_{i=-L+1}^{L} \P^{\lambda,0}_{\omega,(i,\bar x)} \bigl(X_1=(i-1,\bar x)\bigr) \left(\bar u(i,\bar x)-\bar u(i-1,\bar x)\right)^2 \\
    &\le& \tilde C \sum_{\bar x \in \Z^{d-1}, \, \|\bar x\|\le 2\tilde L} \sum_{i=-L}^{L-1} \left(\bar u(i+1,\bar x)-\bar u(i,\bar x)\right)^2 e^{2\lambda_1 i}
\end{eqnarray*}
for some constant $\tilde{C}$. But
\begin{eqnarray*}
    \bar u(i+1,\bar x)- \bar u(i,\bar x) &=& \frac{e^{2L\lambda_1}- e^{-2(i+1)\lambda_1} - e^{2L\lambda_1} + e^{-2i\lambda_1}}{e^{2L\lambda_1}-e^{-2L\lambda_1}} \\
    &=& \frac{e^{-2i\lambda_1}(e^{2\lambda_1}-1)}{e^{2L\lambda_1}-e^{-2L\lambda_1}}.
\end{eqnarray*}
So we have that
\begin{eqnarray*}
    \sum_{i=-L}^{L-1}e^{2\lambda_1 i}\left(\bar u(i+1,\bar x)- \bar u(i,\bar x)\right)^2&=&\sum_{i=-L}^{L-1}e^{2\lambda_1 i}\frac{e^{-4i\lambda_1}(e^{2\lambda_1}-1)^2}{(e^{2L\lambda_1}-e^{-2L\lambda_1})^2}\\
    &=& \frac{(e^{2\lambda_1}-1)^2}{(e^{2L\lambda_1}-e^{-2L\lambda_1})^2}\sum_{i=-L}^{L-1}e^{-2i\lambda_1}\\
    &=&\frac{(e^{2\lambda_1}-1)^2}{(e^{2L\lambda_1}-e^{-2L\lambda_1})^2}e^{2\lambda_1 L}\left(\frac{1-e^{-4\lambda_1 L}}{1-e^{-2\lambda_1}}\right)\\
    &\le& C_2 \frac{ e^{2\lambda_1 L} }{ (e^{2L\lambda_1}-e^{-2L\lambda_1})^2}.
\end{eqnarray*}
This yields for some constant $C_3$
\begin{eqnarray*}
    \mathcal E(\bar u,\tilde\Pi) \le C_3 \frac{e^{2\lambda_1 L}\cdot L^{d-1} \cdot e^{L \lambda_1} }{(e^{2\lambda_1 L}-e^{-2\lambda_1 L})^2} = O(e^{-\lambda_1 L}).
\end{eqnarray*}

For $L\to\infty$ we have $\mathcal E(\bar u,\Pi)\longrightarrow 0$.
Now the problem is that it does not share the boundary value with $u$, so we need to interpolate between $\tilde\Pi$ and $\Pi$.
Note the computations also hold for $\mathcal E(\bar u,\Pi+y)$ with $y\cdot e_1=0$ and $y\in\mathbb Z^d$.

In the following $C$ is a constant whose value might change from line to line. Let $\Pi(y)=y+\Pi$ and $\tilde\Pi(y)=y+\tilde\Pi$ for $y\in\mathbb Z^d$.
For $(x_1,\tilde x)\in\tilde\Pi$ define
$d(x_1,\tilde x)=\mathds{1}_{(x_1,\tilde x)\notin\tilde\Pi}\frac{\|\tilde x\|_1-L}{L}$.
Let $E_L=\sup_{y\in\mathbb Z^{d-1}}\mathcal E(u,\tilde\Pi(0,y))$. Note that $E_L<\infty$ because.
\begin{equation*} 
    \omega^\lambda(x,y)\le \kappa e^{\lambda(x+y)\cdot \ell}\ \text{so}\ \mathcal E(u,\tilde\Pi_{(0,y)})\le \kappa 2^d L \tilde{L}^{d-1} e^{2\lambda \tilde{L}}<\infty. 
\end{equation*}
Choose $y\in\mathbb Z^d$ such that $\mathcal E(u,\tilde\Pi(0))\le \mathcal E(u,\tilde\Pi(y))$ and $\mathcal E(u,\tilde\Pi(y))+\mathcal E(\bar u,\tilde\Pi(y))>E_L$.
Which is possible because $E(\bar u,\tilde\Pi(y))$ has the following lower bound independent of $y$:
\begin{equation*} 
    \mathcal E(\bar u,\tilde\Pi(y))\ge \sum_{\substack{x,z\in\tilde\Pi(y)\\(x-z)\cdot e_1=1}}\P_{\omega,x}^{\lambda,0}(X_1=z)\left(\frac{e^{-2 \lambda_1 x\cdot e_1}-e^{- 2 \lambda_1 z\cdot e_1}}{e^{2\lambda_1 L}-e^{-2\lambda_1 L}}\right)^2\ge \P_{\omega,0}^{\lambda,0}(X_1=e_1)\left(\frac{e^{2 \lambda_1}-1}{e^{2\lambda_1 L}-e^{-2\lambda_1 L}}\right)^2>0.
\end{equation*}
From now on, we fix that $y$. Consequently, we write $\tilde\Pi$ for $\tilde\Pi(y)$ and $\Pi$ for $\Pi(y)$. Let
\begin{eqnarray*} 
    v(x)=(1-d(x))\bar u(x)+d(x)u(x)\qquad \forall x\in\Pi. 
\end{eqnarray*}
Note that $v=u$ on $\partial\Pi$.
Then by Dirichlet's principle we have
\begin{eqnarray*} 
    \mathcal E(u,\Pi)\le \mathcal E(v,\Pi). 
\end{eqnarray*}
Further, for $x\sim y$ ($x,y\in\Pi$) we have:
\begin{eqnarray*} 
    v(x)-v(y)&=&(1-d(x))(\bar u(x)-\bar u(y))+d(x)(u(x)-u(y))+(d(x)-d(y))(u(y)-\bar u(y)). 
\end{eqnarray*}
Then using that for $x\sim y$
$|d(x)-d(y)|\le \tilde L^{-1}\mathds{1}_{\{x,y\in\tilde\Pi^c\}}$
\begin{eqnarray*} 
    (v(x)-v(y))^2&\le&\left[(1-d(x))(\bar u(x)-\bar u(y))+d(x)(u(x)-u(y))\right]^2\\ 
    &&+\tilde L^{-2}\mathds{1}_{\{x,y\in\tilde\Pi^c\}}\left(\bar u(y)-u(y)\right)^2\\ 
    &&+2\tilde L^{-1}\mathds{1}_{\{x,y\in\tilde\Pi^c\}}\left|\bar u(y)-u(y)\right|\left[(1-d(x))|\bar u(x)-\bar u(y)|+d(x)|u(x)-u(y)|\right]\\ 
    &\le&(1-d(x))(\bar u(x)-\bar u(y))^2+d(x)(u(x)-u(y))^2+\tilde L^{-2}\mathds{1}_{\{x,y\in\tilde\Pi^c\}}(\bar u(y)-u(y))^2\\
    &&+2\tilde L^{-1}\mathds{1}_{\{x,y\in\tilde\Pi^c\}}|\bar u(y)-u(y)|\left[(1-d(x))|\bar u(x)-\bar u(y)|+d(x)|u(x)-u(y)|\right],
\end{eqnarray*}
where we used Jensen's inequality for the last inequality. By multiplying both sides by $\omega^\lambda(x,y)$ and summing over all $x,y\in\Pi$, we obtain
\begin{eqnarray} \label{eq1}
    \mathcal E(v,\Pi)&\le&\mathcal E(\bar u,\Pi)+\mathcal E(u,\Pi\setminus\tilde\Pi^+)\\ 
    &&+\tilde L^{-2}\sum_{\substack{z\in\Pi\setminus\tilde\Pi\\x\sim z}}\omega^\lambda(x,z)\left(\bar u(z)-u(z)\right)^2 \label{a}\\
    &&+2\tilde L^{-1}\sum_{\substack{z\in\Pi\setminus\tilde\Pi\\x\sim z}}\omega^\lambda(x,z)|\bar u(z)-u(z)|\left(|\bar u(x)-\bar u(z)|+|u(x)-u(z)|\right), \label{b}
\end{eqnarray}
with $\tilde\Pi^+=\tilde\Pi\setminus\partial\tilde\Pi$. We now want to bound the terms \eqref{a} and \eqref{b}.
\begin{eqnarray*} \eqref{a} &\le&2\tilde L^{-2}\sum_{\substack{z\in\Pi\setminus\tilde\Pi\\x\sim z}}\omega^\lambda(x,z)\left[(u(z)-1)^2+(\bar u(z)-1)^2\right]. \end{eqnarray*}
But for $z=(z_1,\bar z)$ we can apply the Cauchy-Schwarz inequality to have that:
\begin{eqnarray*} (u(z)-1)^2&=&\left(\sum_{i=z_1}^{L}u(i,\bar z)-u(i+1,\bar z)\right)^2\\ &\le&2L\sum_{i=z_1}^{L}\left(u(i,\bar z)-u(i+1,\bar z)\right)^2. 
\end{eqnarray*}
Bounding $(\bar u(z)-1)^2$ in the same way we get
\begin{eqnarray*} 
    \eqref{a}&\le&2\tilde L^{-2}\cdot L\sum_{\substack{z\in\Pi\setminus\tilde\Pi\\z\sim x}}\omega^\lambda(x,z)\left(\sum_{i=z_1}^{L}(u(i,\bar z)-u(i+1,\bar z))^2+(\bar u(i,\bar z)-\bar u(i+1,\bar z))^2\right)\\
     &\le&(2d \kappa^2 e^{\lambda})\tilde L^{-2}\cdot L\sum_{z\in\Pi\setminus\tilde\Pi}\sum_{i=z_1}^{L}\omega^\lambda((i,\bar z),(i+1,\bar z))\left[(u(i,\bar z)-u(i+1,\bar z))^2+(\bar u(i,\bar z)-\bar u(i+1,\bar z))^2\right]\\ 
     &\le&C\tilde L^{-2}L^2\left[\mathcal E(u,\Pi\setminus\tilde\Pi^+)+\mathcal E(\bar u,\Pi\setminus\tilde\Pi^+)\right]\\ 
     &\le&C\tilde L^{-1}L\left[\mathcal E(u,\Pi\setminus\tilde\Pi^+)+\mathcal E(\bar u,\Pi\setminus\tilde\Pi^+)\right]. 
\end{eqnarray*}
For a constant $C$ depending on $d, \kappa$ and $\lambda$.
For the summand \eqref{b} using Hölder's inequality we get that,
\begin{eqnarray*} 
    \eqref{b}&\le&2\tilde L^{-1}\left(\sum_{\substack{z\in\Pi\setminus\tilde\Pi\\x\sim z}}\omega^\lambda(x,z)(u(z)-\bar u(z))^2\right)^{1/2}\left(\sum_{\substack{z\in\Pi\setminus\tilde\Pi\\x\sim z}}\omega^\lambda(x,z)\left[(u(x)-u(z))^2+(\bar u(x)-\bar u(z))^2\right]\right)^{1/2}. 
\end{eqnarray*}
The first sum is the same as the one in (a) so we can bound it by $CL^2\left(\mathcal E(u,\Pi\setminus\tilde\Pi^+)+\mathcal E(\bar u,\Pi\setminus\tilde\Pi^+)\right)$ it follows
\begin{eqnarray*} 
    \eqref{b}\le C\tilde L^{-1}L\left(\mathcal E(u,\Pi\setminus\tilde\Pi^+)+\mathcal E(\bar u,\Pi\setminus\tilde\Pi^+)\right). 
\end{eqnarray*}
Putting the bounds for \eqref{a} \& \eqref{b} together and using that \[
\mathcal{E}(u,\Pi)-\mathcal{E}(u,\Pi\setminus\tilde{\Pi}^{+})=\mathcal{E}(u,\tilde{\Pi})
\qquad\text{and}\qquad
\mathcal{E}(u,\Pi)\leq \mathcal{E}(v,\Pi)
\]
we get:
\begin{eqnarray*}
\mathcal{E}(u,\Pi)\leq \mathcal{E}(v,\Pi)\leq \mathcal{E}(\bar{u},\Pi)+\mathcal{E}(u,\Pi\setminus\tilde{\Pi}^{+})+C\tilde{L}^{-1}L\bigl(\mathcal{E}(u,\Pi\setminus\tilde{\Pi}^{+})+\mathcal{E}(\bar{u},\Pi\setminus\tilde{\Pi}^{+})\bigr).
\end{eqnarray*}
So
\begin{eqnarray}\label{eq2}
\mathcal{E}(u,\tilde{\Pi})&\leq& \mathcal{E}(\bar{u},\Pi)+C\tilde{L}^{-1}L\bigl(\mathcal{E}(u,\Pi\setminus\tilde{\Pi}^{+})+\mathcal{E}(\bar{u},\Pi\setminus\tilde{\Pi}^{+})\bigr)\\
&\leq& \mathcal{E}(\bar{u},\Pi)+C\tilde{L}^{-1}L\bigl(\mathcal{E}(u,\Pi\setminus\tilde{\Pi}^{+})+\mathcal{E}(\bar{u},\Pi)\bigr).
\end{eqnarray}
Since $\Pi$ and $\tilde{\Pi}$ are of comparable size and the energy over a box with size $\tilde{\Pi}$ is bounded by $E_L$ we have
\begin{eqnarray}\label{eq3}
\mathcal{E}(u,\Pi\setminus\tilde{\Pi}^{+})\leq 3^{d-1}E_L\leq 3^{d-1}\bigl(\mathcal{E}(\bar{u},\tilde{\Pi})+\mathcal{E}(u,\tilde{\Pi})\bigr)
\end{eqnarray}
using the definition of $\tilde{L}$ and \eqref{eq2}and \eqref{eq3}
\begin{eqnarray*}
\mathcal{E}(u,\tilde{\Pi})\leq \mathcal{E}(\bar{u},\Pi)+Ce^{-(L\lambda_1 / d)}\bigl(\mathcal{E}(u,\tilde{\Pi})+\mathcal{E}(\bar{u},\tilde{\Pi})\bigr)
\end{eqnarray*}
\begin{eqnarray*}
\leq \mathcal{E}(\bar{u},\Pi)+ce^{-(L\lambda_1 / d)}\bigl(\mathcal{E}(u,\tilde{\Pi})+\mathcal{E}(\bar{u},\Pi)\bigr).
\end{eqnarray*}
However, we have shown that $\mathcal{E}(\bar{u},\Pi)\stackrel{L\to\infty}{\to}0$ so $\lim_{L\to\infty}\mathcal{E}(u,\tilde{\Pi})=0$. This is a contradiction to the lower bound we had compute for $\mathcal{E}(u,\Pi_{\ell})$ with $\Pi_{\ell}$ a non-empty rectangle, that yields $\mathcal{E}(u,\tilde{\Pi})\geq c_1>0$ for all $L$.

\section{Proof of Lemma \ref{deathbirthchain}}\label{sect:Lemma}
We want to prove Lemma \ref{deathbirthchain}, first we need a lemma about hitting time of random walks.
\begin{lem}\label{LemmaHittingTimeRW}
    Let $L\in \mathbb{N}$.
    Let $(X_n)_{n \geq 0}$ be a random walk with $X_n = X_{n-1} + \xi_n$, where $(\xi_i)_{i \geq 0}$ is a sequence of i.i.d. random variables and $\P(\xi_1 = -1) = 1-\P(\xi_1 = L) = p$ with $p$ such that $\E[\xi_1] < 0$.
    Then the random variable $H = \inf\{n \geq 0 : X_n - X_0 = -L\}$ has an exponential tail and, in particular, $\E[H] < \infty$.
\end{lem}
\begin{proof}
    $X_n - X_0 = \sum_{i = 1}^{n} \xi_i$ is a sum of i.i.d. random variables. Further we see that for $\lambda \in \mathbb{R}$ $\E[e^{\lambda \xi}] = p e^{-\lambda} + (1-p) e^{L\lambda} < \infty$.Let $y = \E[\xi_1] /2$. For $n \geq \frac{L}{y}$,
    \[
        \P(H \geq n+1) \leq \P(Z_n-Z_0 > -L)\leq \P(Z_n-Z_0 > ny) \leq e^{-n I(y)},
    \]
    with $I(y) = \sup_{\lambda \in \mathbb{R}}(\lambda y - \log(\E[e^{\lambda \xi_1}]))$.\\
    To see that $I(y) \neq 0$ we take $\lambda^* = \frac{1}{L+1} \log\left(\frac{p(1+y)}{(L-y)(1-p)}\right)$, and we see that $\lambda^* y - \log(\E[e^{\lambda^* \xi_1}]) >0$.\\
    So $\E[H] = \sum_{i = 0}^{\infty} \P(H \geq i) < \infty$ and $H$ has an exponential tail.\\
\end{proof}
\begin{proof}[Proof of Lemma \ref{deathbirthchain}]\label{deathbirthchainproof}
    As the transition rates  of $(A_t)_{t \geq 0}$ are bounded away from zero we can look at the embedded discrete Markov chain $(A_n)_{n \geq 0}$ instead of the continuous time Markov process.
    
    Let $x > \alpha L (1 + \mu^{-1})$, $T^x_1 = \inf\{n \geq 1: A_n = x\}$ and $T^x_{k+1} = \inf\{n >T^x_k: A_n = x\}$ for $k \geq 1$. For easier notation we write $T^x_0 = 0 $. Further we define $N = \max\{k \geq 0 : T^x_k < \tau_1\}$.\\
    We now can write 
    \begin{eqnarray*}
        \tau_1 = \sum_{k = 0}^{N-1} \left(T^x_{k+1} - T^x_{k}\right) + \left(\tau_1 - T^x_N\right).
    \end{eqnarray*}
    First we see that $\left(\tau_1 - T^x_N\right)$ can be expressed as a hitting time (from $x$ to $0$ if $N \geq 1$ or from $0$ to $0$ if $N=0$) on a finite irreducible Markov chain. So there exists an $\varepsilon_1$ such that for all $\varepsilon < \varepsilon_1$,\; $\E[e^{\varepsilon\left(\tau_1 - T^x_N\right)}] < \infty$.\\
    Next we define for $k \geq 0$  $T^{\geq x}_{k} = \inf\{n >T^x_k: A_n \geq x\}$. Note that $T^{x}_{k} < T^{\geq x}_{k} \leq T^{x}_{k+1}$. First we see that on the event $\{k < N\}$, $T^{\geq x}_{k} -T^{ x}_{k}$ can be seen as a hitting time on a finite Markov chain, so there is an $\varepsilon_2$ such that for all $\varepsilon < \varepsilon_2$,\; $\E[e^{\varepsilon\left(T^{\geq x}_{k} -T^{ x}_{k}\right)}| k < N] < \infty$.\\
    Next we look at $T^{ x}_{k+1}-T^{\geq x}_{k}$. We can couple between $T^{\geq x}_{k}$ and $T^{ x}_{k+1}$ $(A_n)_{n \in \{T^{\geq x}_{k},...,T^{ x}_{k+1}\} }$ to a random walk as defined in Lemma \ref{LemmaHittingTimeRW}. Let $H$ be as in Lemma \ref{LemmaHittingTimeRW}  then we get that $T^{ x}_{k+1}-T^{\geq x}_{k} \leq H$. But we know that $H$ has some exponential moments so there exists an $\varepsilon_3>0$ such that for all $\varepsilon < \varepsilon_3$,\; $\E[e^{\varepsilon\left(T^{ x}_{k+1}-T^{\geq x}_{k}\right)}] < \infty$.\\
    So for $\varepsilon < \min\{\varepsilon_2,\varepsilon_3\}$ we have $\E[e^{\varepsilon\left(T^{ x}_{k+1}-T^{x}_{k}\right)}| k < N] < \infty$.\\
    Let $p = \prod_{k=1}^{x} \frac{\mu k}{\mu k +\alpha}$ then $\P(N = k) \leq \P(N \geq k) \leq (1-p)^k$. 
    \begin{eqnarray*}
        \E[\prod_{k = 0}^{N-1} e^{\varepsilon\left(T^{ x}_{k+1}-T^{x}_{k}\right)}] &=& \sum_{n = 0}^{\infty} \E[\prod_{k = 0}^{N-1} e^{\varepsilon\left(T^{ x}_{k+1}-T^{x}_{k}\right)}|N = n] \P(N = n)\\
        &\leq& \sum_{n = 0}^{\infty} \E[\prod_{k = 0}^{n-1} e^{\varepsilon\left(T^{ x}_{k+1}-T^{x}_{k}\right)}|N = n] (1-p)^n\\
        &\leq& \sum_{n = 0}^{\infty} \E[e^{\varepsilon\left(T^{ x}_{1}-T^{x}_{0}\right)}|N > 0]^n (1-p)^n.\\
    \end{eqnarray*}
    By dominated convergence we get that $\lim_{\varepsilon \rightarrow 0 } \E[e^{\varepsilon\left(T^{ x}_{1}-T^{x}_{0}\right)}|N > 0] = 1$, so we can take $\varepsilon_4$ small enough to get that for all $\varepsilon < \varepsilon_4$ \[ \E[e^{\varepsilon\left(T^{ x}_{1}-T^{x}_{0}\right)}|N > 0] (1-p) < 1\] and \[\E[\prod_{k = 0}^{N-1} e^{\varepsilon\left(T^{ x}_{k+1}-T^{x}_{k}\right)}] < \infty.\] 
    Altogether we have that for $\varepsilon < \min\{\varepsilon_1, \varepsilon_4\}$, 
    \[\E[e^{\varepsilon \tau_1}] = \E[e^{\varepsilon\left(\tau_1 - T^x_N\right)}] \E[\prod_{k = 0}^{N-1} e^{\varepsilon\left(T^{ x}_{k+1}-T^{x}_{k}\right)}] < \infty.\]
\end{proof}

\section{CBRW not a time change of VBRW}
\label{sect:Counter}
Let us consider the network $\{0, 1, 2\}$ endowed with nearest-neighbor edges and random conductances i.i.d.\ with the distribution
\begin{equation*}
	\omega_0(\{0, 1\})  \sim \frac{1}{2}\delta_{M} + \frac{1}{2}\delta_{1},
\end{equation*}
with $M \gg 1$ that will be fixed later. Concretely, the random conductances under the measure $\mathbf{P}$ are i.i.d.\ Bernoulli $1/2$ with support $\{1, M\}$. Moreover, the conductances are re-sampled according to i.i.d.\ Poisson processes with intensity $\mu>0$ independently of everything else. For simplicity of the exposition we consider $\mu = 1$. 

Let $P^{\omega_t, \mathrm{VS}}$ be the quenched law of the \textbf{variable speed} walk in the random walk in the dynamic environment. This is the Markov chain that jumps from $x\in \mathbb{Z}^d$ to a neighbor $y$ at time-inhomogeneous rate $\omega_t(x, y)$. Let $\mathbb{P}^{\mathrm{VS}} = \mathbf{E}[P^{\omega_t, \mathrm{VS}}]$ be the annealed law.

Similarly, we let $P^{\omega_t, \mathrm{CS}}$ be the quenched law of the \textbf{constant speed} walk in the random walk in the dynamic environment. Let $\mathbb{P}^{\mathrm{CS}} = \mathbf{E}[P^{\omega_t, \mathrm{CS}}]$ be the annealed law.

We claim that one cannot find a (possibly random) time-change $A_t$ such that
\begin{equation*}
	X^{\mathrm{CS}}_{A_t} = X^{\mathrm{VS}}_{t}.
\end{equation*}
To this end we construct an event that would not be affected by the time change and show that it has different probabilities in the two models.
\begin{equation*}
	E \coloneqq \left\{X \text{ hits } 0 \text{ }\sqrt{M} \text{ times before hitting }2\right\}.
\end{equation*}
This event is not affected by any monotone increasing time change.

\textbf{Positive probability under VBRW}. We observe that the probability that the initial configuration is $\omega_0(\{0, 1\}) = M, \omega_0(\{1, 2\}) = 1$ is exactly $1/4$. Let us denote this event $A_1$ and let us condition on this event. Let $e_1, e_2$ be the exponential random variables associated to the edges. The event $e_1 > 1, e_2>1$, this event also has a strictly positive probability (meaning, independent of $M$), say $p>0$. Let us denote this event $A_2$ and let us condition on this event.

Observe that in the time interval $[0, 1]$ the edge with conductance $M$ has $\mathrm{Pois}(M)$ many rings (to which random walk jumps correspond) and the edge with conductance $M$ has $\mathrm{Pois}(1)$ many.

Thus, we obtain
\begin{align}
	\mathbb{P}^{\mathrm{VS}}(E) &\ge \mathbf{P}(A_1) \mathbf{P}(A_2) P^\omega_t\left(\mathrm{Pois}(M) > 2\sqrt{M},\mathrm{Pois}(1) = 0\right)\\
	& \ge \frac{p}{4e}P^\omega_t\left(\mathrm{Pois}(M) > 2\sqrt{M}\right) = \frac{p}{4e}\left(1 - g(M)\right),\label{eqn:CSLB}
\end{align}
with $g(M) \to 0$ as $M \to \infty$. Indeed, for the last step we observe that 
\begin{equation*}
	P^\omega_t\left(\mathrm{Pois}(M) \le 2\sqrt{M}\right) = \sum_{j = 0}^{2\sqrt{M}} e^{-M} \frac{M^j}{j!}.
\end{equation*}
Let $k = 2\sqrt{M}$, then we observe that 
\begin{equation*}
	\frac{M^{j-1}}{(j-1)!} \frac{j!}{M^j} \le \frac{k}{M} \le \frac{1}{\sqrt{M}},
\end{equation*}
thus the terms are decreasing and hence  
\begin{equation*}
	\sum_{j = 0}^{k} e^{-M} \frac{M^j}{j!} \le e^{-M}k \frac{M^{k}}{k!},
\end{equation*}
applying Stirling's formula, we obtain 
\begin{equation*}
	\frac{M^{k}}{k!} \le \left(\frac{eM}{k}\right)^k \le e^{\log \left( \left(\frac{eM}{k}\right)^k\right)} \le e^{k \log(M)}.
\end{equation*}
Which means that the term $e^{-M}$ dominates whenever $M$ is large enough.

The key point here is that the lower bound for $\mathbb{P}^{\mathrm{VS}}(E)$ computed at \eqref{eqn:CSLB} stays bounded away from $0$ as $M \to \infty$.

\textbf{Decaying probability under CBRW}.

We aim to show that for all $t$ and all environment configurations $\{\omega_s\}_{s \le t}$ if $X_t^{\mathrm{CS}}$ is at $1$ then with positive probability
\begin{equation*}
	\mathbb{P}_1^{\mathrm{CS}}\left( X_{T_t} = 2 \mid \{\omega_s\}_{s \le t}, \{X_s\}_{s \le t}\right) \ge p_1 > 0,
\end{equation*}
where $T_t$ is the first jump after time $t$. Indeed, let $A_3$ be the event that both edges refresh before time $1$ while the random walk jump happens after time $1$. Since here we consider the constant speed random walk this has strictly positive probability (regardless of $M$). Moreover, by memoryless property (of both the clocks of the environment and the walk) this event has the same probability (say $p_2$) after any time-shift by $t>0$.

Furthermore, we observe that whenever $A_3$ happens then with environment probability $3/4$ the random walk has probability at least $1/2$ to jump to $2$ before jumping to $1$.

But then, by the strong Markov property of $\omega_s$ and of the quenched random walk, we obtain that
\begin{equation*}
	\mathbb{P}_1^{\mathrm{CS}}(E) \le \left( 1 - \frac{3p_1}{8}\right)^{\sqrt{M}} \le e^{-c \sqrt{M}}.
\end{equation*}

We conclude by observing that the constant speed and the variable speed walk give different probabilities to the event $E$ that is not affected by the time change. Hence, we cannot express $X^{\mathrm{VS}}$ as a (random) time-change of $X^{\mathrm{CS}}$.

\printbibliography
\end{document}